\documentclass[a4paper,10pt]{article}
\usepackage[utf8]{inputenc}

\usepackage{amsmath, amssymb, amsthm}
\usepackage[utf8]{inputenc}
\usepackage{alltt}
\usepackage{amscd}
\usepackage{url}
\usepackage{setspace}
\usepackage{xcolor}
\usepackage{tikz}
\usepackage{float}
\usepackage{subfloat}
\usepackage{MnSymbol}
\usepackage{paralist}
\usepackage{stmaryrd}
\usepackage[all]{xy}

\usetikzlibrary{matrix}

\newtheorem{defn}{Definition}[section]
\newtheorem{prop}[defn]{Proposition}
\newtheorem{cor}[defn]{Corollary}
\newtheorem{lemma}[defn]{Lemma}

\newtheorem{constr}[defn]{Construction}
\newtheorem{thm}[defn]{Theorem}
\newtheorem{ex}[defn]{Example}

\theoremstyle{remark}
\newtheorem{rem}[defn]{Remark}

\newcommand{\NN}{\mathbb{N}}
\newcommand{\ZZ}{\mathbb{Z}}
\newcommand{\CC}{\mathbb{C}}
\newcommand{\QQ}{\mathbb{Q}}
\newcommand{\bAA}{\mathbb{A}}

\newcommand{\bSS}{\mathbb{S}}
\newcommand{\PP}{\mathbb{P}}

\newcommand{\T}{\mathcal{T}}
\newcommand{\U}{\mathcal{U}}

\newcommand{\V}{\mathcal{V}}
\newcommand{\F}{\mathcal{F}}
\newcommand{\G}{\mathcal{G}}
\newcommand{\cH}{\mathcal{H}}

\newcommand{\M}{\mathcal{M}}
\newcommand{\E}{\mathcal{E}}
\newcommand{\cO}{\mathcal{O}}

\newcommand{\A}{\mathcal{A}}
\newcommand{\C}{\mathcal{C}}

\newcommand{\D}{\mathcal{D}}
\newcommand{\I}{\mathcal{I}}

\newcommand{\m}{\mathfrak{m}}

\newcommand{\Spec}{\mathrm{Spec}\ }

\newcommand{\lookup}{\colorbox{yellow}{Look it up!}}

\newcommand{\TW}{\operatorname{TW}}
\newcommand{\PV}{\operatorname{PV}}

\title{Log Smooth Deformation Theory \\ via Gerstenhaber Algebras}
\author{Simon Felten\footnote{S.~Felten, Math.~Institut, Johannes-Gutenberg-Universit\"at Mainz, 55128 Mainz, Germany, email:  felten.math@posteo.net \newline
}}

\begin{document}

\maketitle

\begin{abstract}
 We construct a $k\llbracket Q\rrbracket$-linear predifferential graded Lie algebra $L^\bullet_{X_0/S_0}$ associated to a log smooth and saturated morphism $f_0: X_0 \to S_0$ and prove that it controls the log smooth deformation functor. This provides a geometric interpretation of a construction in \cite{ChanLeungMa2019} whereof $L^\bullet_{X_0/S_0}$ is a purely algebraic version. Our proof crucially relies on studying deformations of the Gerstenhaber algebra of polyvector fields; this method is closely related to recent developments in mirror symmetry.
\end{abstract}

\section{Introduction}

Given a smooth variety $X$ over an algebraically closed field $k \supset \QQ$, smooth deformation theory associates a functor $\mathrm{Def}_X: \mathbf{Art}\to \mathbf{Set}$ of Artin rings to it such that $\mathrm{Def}_X(A)$ is the set of isomorphism classes of smooth deformations of $X$ over $\Spec A$.

To understand the properties of such functors of Artin rings, differential graded Lie algebras (dglas) are well-established. Given a dgla $L^\bullet$, there is an associated functor 
$$\mathrm{Def}_{L^\bullet}: \mathbf{Art} \to \mathbf{Set}$$ 
of Artin rings by taking an Artin ring $A$ to solutions $\eta \in L^1 \otimes \m_A$ of the Maurer--Cartan equation
$$d\eta + \frac{1}{2}[\eta,\eta] = 0$$
up to gauge equivalence. We say a dgla $L^\bullet$ controls deformations of $X$ if $\mathrm{Def}_X \cong \mathrm{Def}_{L^\bullet}$. Such an isomorphism reduces the problem of understanding deformations of $X$ to the problem of understanding $L^\bullet$. This has been fruitful in various places, e.g.~when proving the Bogomolov--Tian--Todorov theorem with purely algebraic means in \cite{AlgebraicBTT2010}.

\subsubsection*{Our Setup: Log Smooth Deformation Theory}

In this paper, we establish an analog of the above dgla in logarithmic geometry, replacing the smooth variety $X$ by a log smooth morphism. Log smoothness is the analog of the classical notion of smoothness and has been introduced in \cite{kkatoFI}. It allows to treat certain singularities as if they were smooth, a technique which has been proven fruitful in smoothing normal crossing spaces, e.g.~by Kawamata--Namikawa in \cite{KawamataNamikawa1994}, and in mirror symmetry, e.g.~by Gross--Siebert in \cite{Gross2011}. Infinitesimal log smooth deformation theory was established by F. Kato in \cite{Kato1996}. It studies deformations of a log smooth morphism $f_0: X_0 \to S_0$ of fs log schemes. Here, the base is a log point $S_0 = \Spec(Q \to k)$ for a sharp toric monoid $Q$, most commonly $Q = 0$ or $Q = \NN$. The map $Q \to k$ is $0 \mapsto 1$ and $0 \not= q \mapsto 0$. Log smooth deformation theory replaces $\mathbf{Art}$ by $\mathbf{Art}_Q$, the category of local Artin $k\llbracket Q \rrbracket$-algebras with residue field $k$. Every $A \in \mathbf{Art}_Q$ gives rise to a log ring $Q \to A$ via the $k\llbracket Q\rrbracket$-algebra structure, and thus to a log scheme $S = \Spec (Q \to A)$, for which we write $S \in \mathbf{Art}_Q$ by abuse of notation. A \emph{log smooth deformation} over $S \in \mathbf{Art}_Q$ is a Cartesian diagram 
\[
 \xymatrix{
  X_0 \ar[r] \ar[d]^{f_0} & X \ar[d]^f \\
  S_0 \ar[r] & S \\
 }
\]
where $f$ is log smooth. Because $S_0 \to S$ is strict, it does not matter in which category of log schemes---of all, of fine, or of fine saturated log schemes---we consider Cartesianity. For technical reasons, we restrict to morphisms $f_0: X_0 \to S_0$ which are saturated, cf.~Remark~\ref{sat-dense-rem}. This notion is also known as being of Cartier type; most strikingly, saturated and log smooth implies that the fibers of $f_0$ are reduced and Cohen--Macaulay fibers, cf.~\cite{LoAG2018} for more of its elementary properties. In case $f_0$ is saturated, the deformation $f$ is saturated as well. \'Etale locally, such deformations exist uniquely up to (non-unique) isomorphism, cf.~\cite[Prop.~8.4]{Kato1996}. Taking isomorphism classes of deformations yields the log smooth deformation functor
$$LD_{X_0/S_0}: \mathbf{Art}_Q \to \mathbf{Set},$$
which has a representable hull in case $f_0: X_0 \to S_0$ is proper by \cite[Thm.~8.7]{Kato1996}. 

\subsubsection*{The Main Result: A Pre-DGLA Controls $LD_{X_0/S_0}$}

We provide the dgla in the log setting by translating ideas of Chan--Leung--Ma in \cite{ChanLeungMa2019} to our setting and bridging the remaining gap to log smooth deformation theory. To achieve this, $k\llbracket Q\rrbracket$-linear dglas do not suffice. Namely, if $LD_{X_0/S_0}$ were controlled by such a dgla, then for every $S \in \mathbf{Art}_Q$, we would have a deformation over $S$ corresponding to the trivial Maurer--Cartan solution $\eta = 0$, so there would exist a log smooth deformation over every base (which is in the classical setting the fiber product $X \times S$). This is not the case---the spaces in \cite[Thm.~3]{Persson1983} are $d$-semistable normal crossing spaces all of whose flat deformations are locally trivial. Thus, by \cite[Thm.~11.7]{Kato1996}, they can be endowed with a log smooth log structure over $\Spec(\NN \to k)$, but they admit no log smooth deformation over e.g.~$\Spec(\NN \to k[t]/(t^2))$. Instead of dglas, we employ $k\llbracket Q\rrbracket$-linear \emph{pre}differential graded Lie algebras, a notion which essentially coincides with almost dgla in \cite{ChanLeungMa2019}. We define them and study systematically their basic properties in the Appendix~\ref{pdgla}. They are graded Lie algebras $(L^\bullet,[-,-])$ over $k\llbracket Q\rrbracket$ endowed with a derivation $d$ which does not need to be a differential, but which admits an element $\ell \in L^2$ with $d^2 = [\ell,-]$. This allows us to define an associated deformation functor 
$$\mathrm{Def}_{L^\bullet}: \mathbf{Art}_Q \to \mathbf{Set}$$
via the modified Maurer--Cartan equation $d\eta + \frac{1}{2}[\eta,\eta] + \ell = 0$;
then $\eta = 0$ is no longer a solution.

\begin{thm}\label{log-smooth-dgla}
 Let $f_0: X_0 \to S_0$ be log smooth and saturated. Then there is a $k\llbracket Q\rrbracket$-linear predifferential graded Lie algebra $L^\bullet := L^\bullet_{X_0/S_0}$ and a natural equivalence 
 $$LD_{X_0/S_0} \cong \mathrm{Def}_{L^\bullet}$$
 of functors of Artin rings.
\end{thm}

Since $\mathrm{Def}_{L^\bullet}$ is a deformation functor with tangent space isomorphic to $H^1(X_0,\Theta^1_{X_0/S_0})$, we also recover the existence of a hull in case $f_0$ is proper. 

In case $Q = 0$, the pdgla $L^\bullet_{X_0/S_0}$ is in fact a dgla. Here Theorem~\ref{log-smooth-dgla} is only a slight generalization of existing results.  E.g.~if $D \subset X$ is a normal crossing divisor in a smooth variety, a dgla controlling divisorial deformations (which correspond to log smooth deformations here) has been studied in \cite{Katzarkov2008} and \cite{Iacono2015}. A variant of this are the compactified Landau--Ginzburg models of \cite{Katzarkov2017}. This paper uses the $L_\infty$-approach, which is related to our dgla by homotopy transfer results like \cite[Thm.~3.2]{AlgebraicBTT2010} going back to \cite{Kadeishvili1982}.

\subsubsection*{The Strategy of the Proof}

A purely algebraic construction of a dgla controlling \emph{smooth} deformations is used in \cite{AlgebraicBTT2010}, but the method is not sufficient to prove Theorem~\ref{log-smooth-dgla} because it relies on regluing the trivial smooth deformation. Instead, we prove the theorem in the following three steps:

$\mathbf{LD} \cong \mathbf{GD}$. We study Gerstenhaber algebras of polyvector fields. Given a deformation $f: X \to S$ of $f_0: X_0 \to S_0$, the exterior powers of the relative log derivations $\Theta^1_{X/S}$ can be endowed with a Schouten--Nijenhuis bracket, giving rise to the Gerstenhaber algebra $G^\bullet_{X/S}$. We view it as a deformation of the Gerstenhaber algebra $G^\bullet_{X_0/S_0}$ associated to the central fiber $f_0$. After fixing an affine cover $\{V_\alpha\}$ of $X_0$, every deformation of $G^\bullet_{X_0/S_0}$ must be locally isomorphic to the Gerstenhaber algebra $G^\bullet_{V_\alpha/S}$ of the unique log smooth deformation of $V_\alpha/S_0$ (which we denote by $V_\alpha/S$ by abuse of notation). It turns out that gluings of the $G^\bullet_{V_\alpha/S}$ (encoded by a deformation functor $GD_{X_0/S_0}$) are equivalent to actual log smooth deformations $f: X \to S$, setting up an equivalence of deformation functors $LD_{X_0/S_0} \cong GD_{X_0/S_0}$. The equivalence heavily depends on a careful study of automorphisms of log smooth deformations, which we carry out in Section~\ref{inf-auto}. 

$\mathbf{GD} \cong \mathbf{TD}$. In the next step, the abstract setup of \cite{ChanLeungMa2019} enters. The Gerstenhaber algebra $G^\bullet_{X_0/S_0}$ corresponds to $0$-th order data in \cite[Defn.~2.9]{ChanLeungMa2019}, and the $G^\bullet_{V_\alpha/S}$ correspond to the higher order data in \cite[Defn.~2.13]{ChanLeungMa2019}. The patching data in \cite[Defn.~2.15]{ChanLeungMa2019} correspond to isomorphisms on $V_{\alpha} \cap V_{\beta}$ which are induced by geometric isomorphisms of log smooth deformations. 

There is no canonical gluing of the $G^\bullet_{V_\alpha/S}$ (that would correspond to something like a trivial deformation). Instead, we perform a Thom--Whitney resolution; it yields a bigraded Gerstenhaber algebra $\TW^{\bullet,\bullet}(G^\bullet_{V_\alpha/S})$ with a non-trivial differential $d: \TW^{p,q} \to \TW^{p,q + 1}$. Once we forget this differential, there is---up to isomorphism---a unique gluing $\PV_{X_0/S}$ of the affine patches $\TW^{\bullet,\bullet}(G^\bullet_{V_\alpha/S})$, essentially corresponding to \cite[Lemma 3.21]{ChanLeungMa2019}. Our proof shows the cohomological principles behind the explicit-constructive proof in \cite[Lemma 3.21]{ChanLeungMa2019}. 

Given a gluing of the $G^\bullet_{V_\alpha/S}$, its Thom--Whitney resolution endows $\PV_{X_0/S}$ with a differential $d: \PV_{X_0/S}^{p,q} \to \PV_{X_0/S}^{p,q+1}$ once we have chosen an isomorphism $\PV_{X_0/S} \cong \TW^{\bullet,\bullet}(G^\bullet_{V_\alpha/S})$ of bigraded sheaves. Conversely, every differential on $\PV_{X_0/S}$ gives a gluing of the $G^\bullet_{V_\alpha/S}$ by taking cohomology. This sets up an equivalence of deformation functors $GD_{X_0/S_0} \cong TD_{X_0/S_0}$ where the latter classifies differentials.

$\mathbf{TD} \cong \mathbf{Def}$. Relaxing the condition $d^2 = 0$ on a differential, we obtain the notion of \emph{predifferential}. We prove that there is a compatible way to endow the $\PV_{X_0/S}$ with a predifferential, corresponding to $\bar\partial_\alpha + [\mathfrak{d}_\alpha,-]$ in \cite[Thm.~3.34]{ChanLeungMa2019}. Once we have a predifferential, the defect of another predifferential to be a differential is measured by a Maurer--Cartan equation, corresponding to the so-called classical Maurer--Cartan equation in \cite[Defn.~5.10]{ChanLeungMa2019}. This sets up the final equivalence $TD_{X_0/S_0} \cong \mathrm{Def}_{L^\bullet}$. The geometric \v{C}ech gluing in \cite[§5.3]{ChanLeungMa2019} loosely corresponds to going back from a Maurer--Cartan solution to glue the $G^\bullet_{V_\alpha/S}$. 

Our Lie algebra $L^\bullet_{X_0/S_0}$ is essentially the $(-1,*)$-part of the differential graded Batalin--Vilkovisky algebra $PV^{*,*}(X)$ in \cite[Thm.~1.1]{ChanLeungMa2019}. What is new in our paper, is that it actually controls log smooth deformations. The main technical difference is that whereas \cite{ChanLeungMa2019} uses a countable covering $\{U_i\}$ of $X_0$ and sticks to sections over these opens, we consequently stick to sheaves on $X_0$. Therefore, we hope to help the algebraic geometer to get through the ideas of \cite{ChanLeungMa2019}.

\subsubsection*{An Example}

In Section~\ref{ex-sec}, we compute and discuss the $k\llbracket \NN\rrbracket$-linear pdgla $L^\bullet_{C_0/S_0}$ of a proper log smooth curve $f_0: C_0 \to S_0$ over $S_0 = \Spec (\NN \to k)$. The computation suggests that, in general, it is very difficult to compute $L^\bullet_{X_0/S_0}$ explicitly.

\subsubsection*{Outlook}

Beyond the scope of this paper, it should be possible to use the techniques of \cite[Thm.~1.2]{ChanLeungMa2019} to prove that in case $f_0: X_0 \to S_0$ is log Calabi--Yau, i.e.~$\Omega^d_{X_0/S_0} \cong \cO_{X_0}$, it is possible to take $\ell = 0$, i.e., $L^\bullet_{X_0/S_0}$ is actually a dgla. This is unobstructedness in the sense of \cite{ChanLeungMa2019}. Namely, then $\eta = 0$ is a Maurer--Cartan element, corresponding to a distinguished deformation over every $S \in \mathbf{Art}_Q$. In particular, one can construct a limit over $k\llbracket Q\rrbracket$ and ask for an algebraization of that formal scheme. However, this does \emph{not} imply unobstructedness in the sense of a Bogomolov--Tian--Todorov theorem. Namely, this would require $LD_{X_0/S_0}(S') \to LD_{X_0/S_0}(S)$ to be surjective for any thickening $S \to S'$ in $\mathbf{Art}_Q$.

When we finished the first version of this article, we hoped to apply Theorem~\ref{log-smooth-dgla} to obtain a Bogomolov--Tian--Todorov theorem in log geometry. Namely, in \cite{AlgebraicBTT2010}, the classical Bogomolov--Tian--Todorov theorem is obtained by proving that the dgla controlling $\mathrm{Def}_X$ is homotopy abelian. In fact it is possible---under suitable hypotheses---to adapt their methods and use the abstract Bogomolov--Tian--Todorov theorem of \cite{AbstractBTT2017} to prove the central fiber $L^\bullet_{X_0/S_0} \otimes_{k\llbracket Q\rrbracket} k$ homotopy abelian. This is now proven in our preprint \cite{FeltenPetracci2020} and shows a log Bogomolov--Tian--Todorov theorem in case $Q = 0$; it generalizes results on the deformation of pairs $(X,D)$ in \cite[Lemma 4.19]{Katzarkov2008} and \cite{Iacono2015}. Similarly, the Bogomolov--Tian--Todorov theorem for compactified Landau--Ginzburg models in \cite{Katzarkov2017} is a variant of this theme. The case $Q \not= 0$, which was completely open when we finished the first version of this article, is now settled in \cite{FeltenPetracci2020} with an argument in local algebra; it is essentially a consequence of the unobstructedness results by Chan--Leung--Ma in \cite{ChanLeungMa2019}.

Our method is not bound to the log smooth case. E.g.~it should also apply to the deformation theory of log toroidal families of \cite{FFR2019}, which are a generalization of log smooth families which does not need to be log smooth everywhere. The key missing step for the general situation is local uniqueness of deformations, which is only known for some types of log toroidal families, see \cite[Thm.~6.13]{FFR2019}. In particular, once the theory is established, we would get the existence of a hull of the log toroidal deformation functor. In less generality, the existence of a hull has been recently achieved by Ruddat--Siebert in \cite[Thm.~C.6]{RS2019} for divisorial deformations of toric log Calabi--Yau spaces. In this paper, we restrict ourselves to the log smooth setting for simplicity.

There has been recent interest in Gerstenhaber algebras of polyvector fields also in tropical geometry. In \cite{Mandel2019}, integral polyvector fields on algebraic tori are studied in order to compute Gromov--Witten invariants via tropical curves. Moreover, it turns out that wall-crossing transformations, i.e., the gluing isomorphisms of the Gross--Siebert program, induce particularly simple operations on these polyvector fields. This gives a new tropical encoding of the Gross--Siebert gluing on the level of Gerstenhaber algebras analogous to our first isomorphism $LD_{X_0/S_0} \cong GD_{X_0/S_0}$. We expect this to be related also to~\cite{LeungMaYoung2019}, where the refined scattering diagrams of~\cite{Mandel2015} are related to Maurer--Cartan solutions in an appropriate dgla.  

In \cite{Kato1998}, F. Kato introduces functors of \emph{log} Artin rings as a more natural framework for infinitesimal log smooth deformation theory. The question which generalized dglas control these functors is subject to future studies.

Finally, it would be interesting how deformation quantization relates to log smooth deformation theory, relating the base $k\llbracket\hbar\rrbracket$ to $k\llbracket\NN\rrbracket$. Deformation quantization of affine toric varieties, which are the building blocks of log smooth morphisms, has been achieved recently in \cite{Filip2018}.

\subsubsection*{Acknowledgement}

This paper is a byproduct which arose from studying \cite{ChanLeungMa2019} in order to apply it in \cite{FFR2019}. I owe many ideas of the present paper to \cite{ChanLeungMa2019}. I thank my collaborator Matej Filip of \cite{FFR2019} for pointing me to the subject and especially the idea to control logarithmic deformations by some sort of dgla.
I thank my PhD advisor Helge Ruddat for encouraging this work and for constant support. I thank Ziming Ma for valuable comments on a first draft. I thank Bernd Siebert for help with the introduction. I thank the anonymous referee for careful reading and many helpful recommendations. Finally, I thank JGU Mainz for its hospitality and Carl-Zeiss-Stiftung for financial support.

\section{Infinitesimal Automorphisms}\label{inf-auto}

The relationship between automorphisms of first order deformations and derivations is well-known. In this section, we determine higher infinitesimal automorphisms of deformations in terms of derivations, i.e., automorphisms of higher order deformations. In \cite{AlgebraicBTT2010}, higher infinitesimal deformations play a crucial role in proving that classical smooth deformations are controlled by a dgla, and we need them as well. Assume we have a surjection $A' \to A$ in $\mathbf{Art}_Q$ with kernel $I \subset A'$ inducing a thickening $S \to S'$, and assume a Cartesian diagram 
\[
 \xymatrix{
 X_0 \ar[r] \ar[d]^{f_0} & X \ar[r] \ar[d]^f & X' \ar[d]^{f'} \\
 S_0 \ar[r] & S \ar[r] & S' \\
 }
\]
of saturated log smooth morphisms. The ideal sheaf of $X \subset X'$ is 
$$\I = \I_{X'/X} = I \cdot \cO_{X'}$$
because $f'$ is flat (as is every integral log smooth morphism). Elements  of the sheaf of automorphisms $\A ut_{X'/X}$ of $X'$ which are compatible with $f': X' \to S'$ and induce the identity on $X$ are pairs $(\phi,\Phi)$ where 
$$\phi: \cO_{X'} \to \cO_{X'} \qquad \mathrm{and} \qquad \Phi: \M_{X'} \to \M_{X'}$$
are the constituting homomorphisms. As we will see below, the sheaf of groups $\A ut_{X'/X}$ is isomorphic to the sheaf $\D er_{X'/S'}(\I)$ of relative log derivations $(D, \Delta)$ with values in $\I$, i.e., $D: \cO_{X'} \to \I$ is a derivation and $\Delta: \M_{X'} \to \I$ is its log part.
This is a sheaf of Lie algebras as a subalgebra of $\Theta^1_{X'/S'}$ which
is filtered by 
$$F^k := F^k\D er_{X'/S'}(\I) := \D er_{X'/S'}(\I^k) \subset \D er_{X'/S'}(\I)$$
for $k \geq 1$, the sheaf of derivations with values in $\I^k \subset \I$.
The ideal $\I$ is generated by elements in the image of 
$f'^{-1}\cO_{S'} \to \cO_{X'}$, so if $(D,\Delta) \in F^k\D er_{X'/S'}(\I)$, then $D(\I^\ell) \subset \I^{k + \ell}$. This shows $[F^k,F^\ell] \subset F^{k + \ell}$, so the 
lower central series of $\D er_{X'/S'}(\I)$ is eventually 
$0$, and it is a sheaf of nilpotent Lie algebras. In particular, following e.g.~\cite{ManettiComplexDefo2004}, 
the Baker--Campbell--Hausdorff formula starting with 
$$\theta * \xi = \theta + \xi + \frac{1}{2}[\theta,\xi] + ...$$
turns $\D er_{X'/S'}(\I)$ into a sheaf of groups.

The sheaves $\D er_{X'/S'}(\I)$ and $\A ut_{X'/X}$ have 
classical analogs, which we denote by $\underline{\D er}_{X'/S'}(\I)$ and $\underline{\A ut}_{X'/X}$. They are the 
classical relative derivations with values in $\I$ respective the automorphisms of the underlying scheme $\underline X'$ over $\underline S'$. Using the Baker--Campbell--Hausdorff formula again, we find $\underline{\D er}_{X'/S'}(\I)$ a sheaf of groups. There are group isomorphisms
\[
 \xymatrix{
  \underline{\D er}_{X'/S'}(\I)
  \ar@<0.5ex>[r]^-{\mathrm{exp}} &
  \underline{\A ut}_{X'/X} \ar@<0.5ex>[l]^-{\mathrm{log}} \\
 }
\]
given by plugging in the derivation $D: \cO_{X'} \to \I$ into the power series expansion of the exponential (at $0$) respective the automorphism $\phi: \cO_{X'} \to \cO_{X'}$ into the expansion of the logarithm (at $1$).

\begin{rem}
 This pair of inverse group isomorphisms seems to be folklore; it is used e.g.~implicitly in \cite[Thm.~5.3]{AlgebraicBTT2010}. However, we are not aware of any explicit reference where it has been proven in the setting of sheaves of automorphisms of flat deformations. That the construction above yields maps between $\underline{\D er}_{X'/S'}(\I)$ and $\underline{\A ut}_{X'/X}$ follows from our explicit computation below; that they are inverse to each other follows from the fact that the two power series $\mathrm{exp}(T)$ and $\mathrm{log}(1 + T)$ are inverse to each other in $\QQ\llbracket T\rrbracket$---this argument has been used e.g.~in \cite[Thm.~7.2]{Serre1964}. The maps are homomorphisms, heuristically, because the Baker--Campbell--Hausdorff formula just makes explicit the product we get when composing the exponentials as endomorphisms.
\end{rem}

We extend the above picture to the diagram 
\begin{equation}\label{inf-auto-diagram}
 \xymatrix{
  \D er_{X'/S'}(\I) \ar@<0.5ex>[r]^-{\mathrm{exp}} \ar[d]^{\iota_D} & 
  \A ut_{X'/X} \ar@<0.5ex>[l]^-{\mathrm{log}} \ar[d]^{\iota_A} \\
  \underline{\D er}_{X'/S'}(\I)
  \ar@<0.5ex>[r]^-{\mathrm{exp}} &
  \underline{\A ut}_{X'/X} \ar@<0.5ex>[l]^-{\mathrm{log}} \\
 }
\end{equation}
where the vertical arrows are the forgetful maps. Given $(D,\Delta) \in \D er_{X'/S'}(\I)$, we define 
$(\phi,\Phi) = \mathrm{exp}_{X'/X}(D,\Delta)$ by the formulas
\begin{align}
 \phi: \cO_{X'} \to \cO_{X'}, \quad \phi(a) &= \sum_{n = 0}^\infty \frac{D^n(a)}{n!}\nonumber \\
 \Phi: \M_{X'} \to \M_{X'}, \quad \Phi(m) &= m + \alpha^{-1}\left(\sum_{n = 0}^\infty \frac{[\Delta(m) + D]^n(1)}{n!}\right) \nonumber 
\end{align}
where the series are actually finite sums.
The symbol $\Delta(m)$ denotes the multiplication operator with this element.

\begin{lemma}
 We have $\mathrm{exp}_{X'/X}(D,\Delta) \in \A ut_{X'/X}$.
\end{lemma}
\begin{proof}
 The map $\phi: \cO_{X'} \to \cO_{X'}$ is a ring automorphism 
 by the method of the classical proof of $e^xe^y = e^{x + y}$. Since 
 $D(f'^{-1}s) = 0$ for $s \in \cO_{S'}$, the morphism 
 $\psi: X' \to X'$ induced by $\phi$ satisfies 
 $f' \circ \psi = f'$, and $\psi \times_{S'} S = \mathrm{id}_X$ because $\phi(a) - a \in \I$. That $\Phi$ is a monoid 
 homomorphism follows from the identity 
 $$\frac{[\Delta(m + m') + D]^n(1)}{n!} 
 = \sum_{k + \ell = n}\frac{[\Delta(m) + D]^k(1)}{k!} \cdot 
 \frac{[\Delta(m') + D]^\ell(1)}{\ell!}$$
 which is proven by induction, and the identity 
 $$D^n(\alpha(m)) = \alpha(m) \cdot [\Delta(m) + D]^n(1)$$
 shows that $\alpha \circ \Phi = \phi \circ \alpha$, i.e., $(\phi,\Phi)$ defines a log morphism $\Psi: X' \to X'$. 
 Because $\Delta(f'^{-1}m) = 0$ for $m \in \M_{S'}$, we have 
 $f' \circ \Psi = f'$, and 
 $$\left(\sum_{n = 0}^\infty \frac{[\Delta(m) + D]^n(1)}{n!}\right) \in 1 + \I$$
 proves that $\Psi \times_{S'} S = \mathrm{id}_X$. The 
 map $\Psi$ is an isomorphism because it is one on underlying schemes and on ghost sheaves, and because $\M_{X'}$ is fine.
\end{proof}

We construct an ansatz for 
the inverse $\mathrm{log}_{X'/X}$. On the classical part,
given $\phi$, our ansatz 
$$ D: \cO_{X'} \to \I, \quad D(a) = \sum_{n = 1}^\infty \frac{(-1)^{n - 1}[\phi - \mathrm{Id}]^n(a)}{n}$$
is of course the formula of the logarithm. The sum is finite 
since $[\phi - \mathrm{Id}](\I^k) \subset \I^{k + 1}$; this identity holds by the explicit formula for $\phi - \mathrm{Id}$ because $D(\I^k) \subset \I^{k + 1}$. 
Given $\Phi$, by induction
 $$[\phi - \mathrm{Id}]^n(\alpha(m)) = \alpha(m) \cdot 
 \left[(-1)^n\sum_{k = 0}^n{n \choose k}(-1)^k\alpha(\Phi^k(m) - m)\right]$$
where $\Phi^k(m) - m \in \M_{X'}^*$ is the unique invertible by which these 
elements differ, so we make an ansatz
$$\Delta: \M_{X'} \to \I, \quad \Delta(m) = \sum_{n = 1}^\infty \left(\sum_{k = 0}^n {{n}\choose{k}} \frac{(-1)^{k + 1}\alpha(\Phi^k(m) - m)}{n}\right)$$
for the log part. Setting $\Phi(m) - m =: v \in \M_{X'}^*$, we find 
$\alpha(v) \in 1 + \I$; indeed, since $\Phi|_X = \mathrm{Id}_X$, we have $v|_X = 0$, and thus $\alpha(v)|_X = 1$. Now inductively 
$$\sigma_n := \sigma_n(\alpha(v)) := (-1)^n\sum_{k = 0}^n\left({{n}\choose{k}} (-1)^k \prod_{i = 0}^{k - 1}\phi^i(\alpha(v))\right) \in \I^{n}$$
since $\sigma_{n + 1} = \alpha(v)\phi(\sigma_n) - \sigma_n$, and we conclude that the sum 
defining $\Delta(m)$ is finite. Indeed, we have $\Delta(m) = \sum_{n = 1}^\infty \frac{(-1)^{n-1}}{n} \sigma_n(\alpha(v))$.

\begin{lemma}
 We have $\mathrm{log}_{X'/X}(\phi,\Phi)\in \D er_{X'/S'}(\I)$ where $\mathrm{log}_{X'/X}$ is defined by the above formulas.
\end{lemma}
\begin{proof}
 Given $(\phi,\Phi)$ let $(D,\Delta)$ be defined by the formulas. Setting $\chi := \phi - \mathrm{Id}$, we have 
 $\chi(\cO_{X'}) \subset \I$. 
 Clearly $D(a + b) = D(a) + D(b)$ and 
 $D(f^{-1}s) = 0$ for $s \in \cO_{S'}$ since 
 $\phi(f^{-1}s) = f^{-1}s$. We have 
 $\chi(ab) = a\chi(b) + b\chi(a) + \chi(a)\chi(b)$ and thus 
 $$\chi^n(ab) = \sum_{k,\ell \geq 0}{n \choose {k,\ell}}\chi^k(a)\chi^\ell(b), \quad
 {n \choose {k,\ell}} := \begin{cases}
                           \frac{n!}{(n - k)!(n - \ell)!(k + \ell - n)!} & 0 \leq k,\ell \leq n \leq k + \ell\\
                           0 & \mathrm{else}\\
                         \end{cases}
$$
for $n \geq 0$ by induction since ${n \choose k-1,\ell} + 
{n \choose k,\ell -1} + {n \choose k-1,\ell - 1} = {n + 1\choose k,\ell}$. Since for a polynomial $p \in \QQ[m]$ of degree $< k$ we have $\sum_{m = 0}^k (-1)^mp(m){k \choose m} = 0$, we find 
$$\gamma_{k,\ell} := \sum_{n = 1}^\infty 
\frac{(-1)^{n - 1}}{n}{n \choose k,\ell} = 
\begin{cases}
 \frac{(-1)^{k + \ell - 1}}{k + \ell} & k\ell = 0, k + \ell \geq 1 \\
 0 & \mathrm{else} \\
\end{cases}$$
and therefore 
$$D(ab) = \sum_{k,\ell \geq 0}\gamma_{k,\ell}\cdot\chi^k(a)\chi^\ell(b) = aD(b) + bD(a)$$
showing that $D: \cO_{X'} \to \I$ is a derivation. For $m_1, m_2 \in \M_{X'}$ writing $\Phi(m_i) = m_i + v_i$ we have 
$$\sigma_n(\alpha(v_1)\alpha(v_2)) = \sum_{k,\ell \geq 0}
{n \choose k,\ell}\sigma_k(\alpha(v_1))\sigma_\ell(\alpha(v_2))$$
by induction starting with $\sigma_0 = 1$, so we get 
$$\Delta(m_1 + m_2) = \sum_{k,\ell \geq 0}\gamma_{k,\ell}\cdot \sigma_k(\alpha(v_1))\sigma_\ell(\alpha(v_2)) 
= \Delta(m_1) + \Delta(m_2)$$
since $\Phi(m_1 + m_2) = m_1 + m_2 + v_1 + v_2$. We have 
$\Delta(f^{-1}n) = 0$ for $n \in \M_{S'}$ since 
$\Phi(f^{-1}n) = f^{-1}n$, and $\alpha \cdot \Delta = D \circ \alpha$ holds by construction.
\end{proof}

The forgetful maps $\iota_D, \iota_A$ in diagram \eqref{inf-auto-diagram} above are injective group homomorphisms. For, if $V' := (X')^{str} \subset X'$ is the strict locus defined in the Appendix, then $\iota_D$ 
and $\iota_A$ are isomorphisms on $V'$. Because
$$\D er_{X'/S'}(\I) \cong \cH om(\Omega^1_{X'/S'},\I)$$
and $\I \subset \cO_{X'}$ has injective restrictions to $V'$ (which is a direct consequence of Proposition~\ref{strict-dense}), we get $\iota_D$ injective. If we have two automorphisms $(\phi,\Phi)$ and $(\phi,\Phi')$, then $\Phi$ and $\Phi'$ coincide on ghost sheaves---this is because the restriction map $\overline\M_{X'} \to \overline\M_X\to \overline\M_{X_0}$ of a log deformation is an isomorphism---so there is an invertible $u$ with 
$\Phi'(m) = u + \Phi(m)$. Because $\iota_A|_{V'}$ is an isomorphism, we have $u|_{V'} = 0$, and thus $u = 0$.
We conclude:

\begin{prop}
 The maps 
 \[
 \xymatrix{
  \D er_{X'/S'}(\I) \ar@<0.5ex>[r]^-{\mathrm{exp}} 
  & \A ut_{X'/X} \ar@<0.5ex>[l]^-{\mathrm{log}} \\
 }
\]
are inverse isomorphisms of sheaves of groups.
\end{prop}
\begin{rem}
 The case of first order deformations, i.e., $\I^2 = 0$, 
 has been described in \cite[Lemma 2.10]{GrossSiebertII}. Indeed, this result has been our main inspiration to find the description of infinitesimal automorphisms of higher order. Automorphisms of all orders of log \emph{rings} have been studied in \cite[Prop.~2.14]{Gross2011}. 
\end{rem}

\section{Gerstenhaber Algebra of Polyvector Fields}\label{Gerstenhaber-algebra-sec}

Given $S \in \mathbf{Art}_Q$, let $f: X \to S$ be log smooth and saturated. The polyvector fields $G^\bullet_{X/S} := \bigwedge^{-\bullet}\Theta^1_{X/S}$ in \emph{negative} grading, i.e., concentrated in degrees $-d \leq \bullet \leq 0$ for $d$ the relative dimension, form a Gerstenhaber algebra. This means it is endowed with two bilinear operations $\wedge$ and $[-,-]$, which are graded in the sense that (for $\G^\bullet = G^\bullet_{X/S}$)
$$\G^p \wedge \G^q \subset \G^{p + q} \quad \mathrm{and} 
\quad [\G^p,\G^q] \subset \G^{p + q + 1}.$$
They satisfy the relations
\begin{itemize}
 \item $x \wedge (y \wedge z) = (x \wedge y) \wedge z
 \quad \mathrm{and} \quad x \wedge y = (-1)^{|x||y|}(y \wedge x)$
 \item $[x,y \wedge z] = [x,y] \wedge z + (-1)^{(|x| + 1)|y|} 
 y \wedge [x,z]$
 \item $[x,y] = - (-1)^{(|x| + 1)(|y| + 1)}[y,x]$
 \item $[x,[y,z]] = [[x,y],z] + 
 (-1)^{(|x| + 1)(|y| + 1)} [y,[x,z]]$
\end{itemize}
where $|x|$ denotes the degree of the homogeneous element $x$. For the bracket $[-,-]$ we take the \emph{negative} $-[-,-]_{sn}$ of the Schouten--Nijenhuis bracket. Recall that the latter one is the unique bracket satisfying the above relations and such that $[g,h]_{sn} = 0$ for functions $g,h \in \cO_X$, such that $[\theta, \xi]_{sn} = [\theta,\xi]$ is the Lie bracket for vector fields $\theta,\xi \in \Theta^1_{X/S}$, and such that  $[\theta,g]_{sn} = \langle dg, \theta\rangle$ for the natural pairing $\langle \cdot,\cdot\rangle$ of vector fields and differential forms. Our grading convention as well as the $(-1)$-sign in the bracket follow \cite{ChanLeungMa2019}.

In the spirit of \cite[Defn.~2.9]{ChanLeungMa2019}, we say a Gerstenhaber algebra $\G^\bullet$ is \emph{$(-1)$-injective} if the map 
$$\G^{-1} \to \cH om(\G^0,\G^0), \quad \theta \mapsto [\theta,-]$$
is injective. Using Proposition~\ref{strict-dense} of the Appendix, we find that $G^\bullet_{X/S}$ is $(-1)$-injective. Indeed, on the strict locus $X^{str} \subset X$ the sheaf $G^{-1}_{X/S}$ represents classical relative derivations on $\cO_X$.

\subsubsection*{The Gerstenhaber Algebra and Deformations}

Let $S \to S'$ be a thickening given by a surjection $A' \to A$ in $\mathbf{Art}_Q$, and let $f': X' \to S'$ be a deformation of $f: X \to S$. Let $\I_{X'/X} = I \cdot \cO_{X'}$ be the ideal sheaf. We obtain an exact sequence 
\begin{equation}\label{rel-def-sequence}
 0 \to \I_{X'/X} \cdot G^\bullet_{X'/S'} \to G^\bullet_{X'/S'} \to G^\bullet_{X/S} \to 0
\end{equation}
of Gerstenhaber algebras. Here the right hand map is induced by pulling back homomorphisms $h: \Omega^1_{X'/S'} \to \cO_{X'}$ along $c: X \to X'$. Elements $\theta \in \I_{X'/X} \cdot G^{-1}_{X'/S'}$ induce \emph{gauge transforms} 
\begin{equation}\label{gauge-def-formula}
 \mathrm{exp}_\theta: G^\bullet_{X'/S'} \to G^\bullet_{X'/S'}, \quad \xi \mapsto \sum_{k = 0}^\infty \frac{([\theta,-])^k(\xi)}{k!},
\end{equation}
which are well-defined Gerstenhaber algebra automorphisms over $G^\bullet_{X/S}$ since $\I_{X'/X}$ is nilpotent. If $\mathrm{exp}_\theta = \mathrm{exp}_\xi$, then $\theta = \xi$ due to $(-1)$-injectivity. Indeed, with the Baker--Campbell--Hausdorff product on $\I_{X'/X} \cdot G^{-1}_{X'/S'}$ the map $\theta \mapsto \mathrm{exp}_\theta$ becomes a group homomorphism, and if $\mathrm{exp}_\theta = \mathrm{Id}$, then $\theta = 0$ by induction since it holds for small extensions $S \to S'$.

Automorphisms of $G^\bullet_{X'/S'}$ can be induced by geometric automorphisms as well. Namely, for two deformations $f'_i: X'_i \to S' \ (i = 1,2)$ and an isomorphism $\varphi: X_1' \to X_2'$ over $f: X \to S$, there is a natural map 
$$T_\varphi: G^\bullet_{X_2'/S'} \to \varphi_*G^\bullet_{X_1'/S'}$$
of Gerstenhaber algebras which is induced by 
$\cO_{X_2'} \to \varphi_*\cO_{X_1'}$ and the pullback of homomorphisms $h: \Omega^1_{X_2'/S'} \to \cO_{X_2'}$ along $\varphi$. More precisely, $h \in \Theta^1_{X_2'/S'}$ gives rise to a composition
$$\Omega^1_{X_1'/S'} \cong \varphi^*\Omega^1_{X_2'/S'} \xrightarrow{\varphi^*h} \varphi^*\cO_{X_2'} \cong \cO_{X_1'},$$
and thus to an element $T_\varphi(h) \in \varphi_*\Theta^1_{X_1'/S'}$. 

\begin{rem}
 The inverse $\varphi^{-1}$ induces a map $d(\varphi^{-1}): \Theta^1_{X_2'/S'} \to (\varphi^{-1})^*\Theta^1_{X_1'/S'}$ whose composition with the canonical map $(\varphi^{-1})^*\Theta^1_{X_1'/S'} \to \varphi_*\Theta^1_{X_1'/S'}$ coincides with $T_\varphi$. In fact, when we write $\Omega_i := \Omega^1_{X_i'/S'}$ and $\cO_i := \cO_{X_i'}$, we have maps $(\varphi^{-1})^*\Omega_1 \to \Omega_2$ and $\Omega_2 \to \varphi_*\Omega_1$; they induce the diagram
 \[
  \xymatrix{
   \cH om(\Omega_2,\cO_2) \ar[r] \ar[dr] & \cH om((\varphi^{-1})^*\Omega_1,(\varphi^{-1})^*\cO_1) & (\varphi^{-1})^*\cH om(\Omega_1,\cO_1) \ar[l] \ar[d] \\
   & \varphi_*\cH om(\varphi^*\Omega_2,\varphi^*\cO_2) \ar[r] & \varphi_*\cH om(\Omega_1,\cO_1) \\
  }
 \]
 of isomorphisms of $\cO_{X_2'}$-modules, which is commutative since the composition $(\varphi^{-1})^*\Omega_1 \to \Omega_2 \to \varphi_*\Omega_1$ is the canonical isomorphism.
\end{rem}

Given a deformation $f': X' \to S'$ of $f: X \to S$, an automorphism $\varphi \in \A ut_{X'/X}$ induces---on the one hand side---a map $T_\varphi: G^\bullet_{X'/S'} \to G^\bullet_{X'/S'}$; on the other hand side, we have an element $\theta = \mathrm{log}_{X'/X}(\varphi) \in \D er_{X'/S'}(\I_{X'/X})$. Now $\D er_{X'/S'}(\I_{X'/X}) = \I_{X'/X}\cdot G^{-1}_{X'/S'}$ because $h: \Omega^1_{X'/S'} \to \cO_{X'}$ takes values in 
 $\I_{X'/X}$ if and only if $c^*(h) = 0$. Thus 
 $$\theta = (D,\Delta) := \mathrm{log}_{X'/X}(\varphi) \in \I_{X'/X}\cdot G^{-1}_{X'/S'}$$
 induces a gauge transform $\mathrm{exp}_{-\theta}$; it coincides with $T_\varphi$ by the Lemma below. The sign in $-\theta$ is due to our convention $[-,-] = -[-,-]_{sn}$.

\begin{lemma}\label{auto-gauge}
 Let $f': X' \to S'$ be a deformation of $f: X \to S$, and let $\varphi \in \A ut_{X'/X}$. Then the induced map 
 $T_\varphi: G^\bullet_{X'/S'} \to G^\bullet_{X'/S'}$ is equal to the gauge transform 
 $\mathrm{exp}_{-\theta}$ for 
 $\theta = \mathrm{log}_{X'/X}(\varphi)$. Moreover, every gauge transform $\mathrm{exp}_\theta$ with $\theta \in \I_{X'/X} \cdot G^{-1}_{X'/S'}$ is induced by a unique automorphism $\varphi \in \A ut_{X'/X}$.
\end{lemma}
\begin{proof}
 For $a \in G^0_{X'/S'}$ we have 
 $$T_\varphi(a) = \sum_{k = 0}^\infty \frac{D^k(a)}{k!} = \sum_{k = 0}^\infty \frac{([-\theta,-])^k(a)}{k!} = \mathrm{exp}_{-\theta}(a)$$
 so it remains to prove equality on $G^{-1}_{X'/S'}$. The derivation $\theta = (D,\Delta)$ induces a map $\nabla: \Omega^1_{X'/S'} \to \cO_{X'} \to \Omega^1_{X'/S'}$ by the formula $\nabla(\omega) = d\langle\omega,\theta\rangle$ that we use in the formula
 $$d\varphi: \quad \Omega^1_{X'/S'} \to \Omega^1_{X'/S'}, \quad \omega \mapsto \sum_{n = 0}^\infty \frac{\nabla^n(\omega)}{n!}$$
 for the action of $\varphi$ on differential forms. Writing $\varphi = (\phi,\Phi)$ as in Section~\ref{inf-auto}, we find for $\xi \in G^{-1}_{X'/S'}$
 $$\langle \omega,T_\varphi(\xi)\rangle = \phi\left(\langle(d\varphi)^{-1}(\omega),\xi\rangle\right) = \langle\omega,\mathrm{exp}_{-\theta}\rangle$$
 by the formula 
 $$\langle \omega,([-\theta,-])^m(\xi)\rangle = \sum_{\ell = 0}^m {{m}\choose{\ell}}(-1)^\ell D^{m - \ell}\langle\nabla^\ell(\omega),\xi\rangle,$$
 which is proven easily by induction. Since $G^\bullet_{X'/S'}$ is (locally) generated by $G^{-1}_{X'/S'}$ as a ring with respect to $\wedge$, we see $T_\varphi = \mathrm{exp}_{-\theta}$. For the second statement, we use that 
 $\mathrm{log}_{X'/X}$ is a bijection and that $G^\bullet_{X'/S'}$ is $(-1)$-injective.
 \end{proof}

\section{Deformations of Gerstenhaber Algebras}\label{def-gerstenhaber-algebra-sec}

Given $f_0: X_0 \to S_0$, we study deformations of the Gerstenhaber algebra $G^\bullet_{X_0/S_0}$, i.e., the deformation functor $GD_{X_0/S_0}$, which we introduce below. We fix an affine cover $\V = \{V_\alpha\}$ of $X_0$ and we denote the---up to non-unique isomorphism---unique log smooth deformation of $V_\alpha$ over $S$ by $V_\alpha \to S$ (by abuse of notation with the same symbol for all $S$). On overlaps, we choose isomorphisms
$$\Phi_{\alpha\beta}: V_\beta|_{V_\alpha \cap V_\beta} \to V_\alpha|_{V_\alpha \cap V_\beta};$$
they induce isomorphisms $\phi_{\alpha\beta}: G^\bullet_{V_\alpha/S}|_{V_\alpha \cap V_\beta} \to G^\bullet_{V_\beta/S}|_{V_\alpha \cap V_\beta}$ of Gerstenhaber algebras \emph{contravariantly} (by our construction in Section~\ref{Gerstenhaber-algebra-sec}). Whenever we can compose these isomorphisms to an automorphism, it is a gauge transform by Lemma \ref{auto-gauge}.

\begin{defn}\label{gerstenhaber-defo-defn}
 A \emph{Gerstenhaber deformation} of $f_0: X_0 \to S_0$ is a 
 Gerstenhaber algebra $\G^\bullet$ on $X_0$ together with a morphism $\G^\bullet \to G^\bullet_{X_0/S_0}$ and isomorphisms $\chi_\alpha: \G^\bullet|_{V_\alpha} \cong G^\bullet_{V_\alpha/S}$ compatible with the map to $G^\bullet_{X_0/S_0}$ and such that on $V_\alpha \cap V_\beta$,
 the cocycle $\phi_{\alpha\beta} \circ \chi_\alpha \circ \chi_\beta^{-1}$ is a gauge transform. An \emph{isomorphism} of Gerstenhaber deformations is an isomorphism $\psi: \G^\bullet_1 \to \G^\bullet_2$ compatible with the maps to 
 $G^\bullet_{X_0/S_0}$ and such that the cocycle $\chi_{2\alpha} \circ \psi \circ \chi_{1\alpha}^{-1}$ is 
 a gauge transform. Isomorphism classes define a functor 
 $$GD_{X_0/S_0}: \mathbf{Art}_Q \to \mathbf{Set}$$
 of Artin rings.
\end{defn}

A log smooth deformation $f: X \to S$ induces a Gerstenhaber deformation $\G^\bullet = G^\bullet_{X/S}$. For $\chi_\alpha$ we take the maps (contravariantly) induced by any geometric isomorphism $X|_{V_\alpha} \cong V_\alpha$. The different choices give rise to isomorphic Gerstenhaber deformations with the identity on $\G^\bullet$ as isomorphism. This induces a 
natural transformation $LD_{X_0/S_0} \Rightarrow GD_{X_0/S_0}$.

\begin{prop}
 The transformation $LD_{X_0/S_0} \Rightarrow GD_{X_0/S_0}$ is an isomorphism of functors.
\end{prop}
\begin{proof}
 We prove $LD_{X_0/S_0}(S) \cong GD_{X_0/S_0}(S)$ for $S \in \mathbf{Art}_Q$. Let $f: X \to S$ and $g: Y \to S$ be two log smooth deformations, choose isomorphisms $X|_{V_\alpha} \cong V_\alpha$ and $Y|_{V_\alpha} \cong V_\alpha$, and suppose there is an isomorphism $\psi: G^\bullet_{X/S} \to G^\bullet_{Y/S}$ of the corresponding Gerstenhaber deformations. It induces gauge transforms $\psi_\alpha: G^\bullet_{V_\alpha/S} \to G^\bullet_{V_\alpha/S}$ according to Definition~\ref{gerstenhaber-defo-defn}, which are induced by automorphisms $\Psi_\alpha: V_\alpha \to V_\alpha$ by Lemma~\ref{auto-gauge}. They give isomorphisms $Y|_{V_\alpha} \to X|_{V_\alpha}$ that glue to a global isomorphism $\Psi: Y \to X$, so the transformation is injective.
 
 Conversely, let $\G^\bullet$ be a Gerstenhaber deformation over $S$, and let $\gamma_{\alpha\beta} = \phi_{\alpha\beta} \circ \chi_\alpha \circ \chi_\beta^{-1}$ be the gauge transform of Definition~\ref{gerstenhaber-defo-defn}. It is induced by an automorphism $\Gamma_{\alpha\beta}: V_\beta|_{V_\alpha \cap V_\beta} \to V_\beta|_{V_\alpha \cap V_\beta}$, which we use to define isomorphisms 
 $$\Psi_{\alpha\beta} := \Phi_{\alpha\beta} \circ \Gamma_{\alpha\beta}^{-1}: V_\beta|_{V_\alpha \cap V_\beta} \to V_\alpha|_{V_\alpha \cap V_\beta} \ .$$ 
 Note that the (contravariantly) induced map on Gerstenhaber algebras is $\chi_\beta \circ \chi_\alpha^{-1}$, so they satisfy the cocycle condition by Lemma~\ref{auto-gauge}. Gluing yields a log smooth deformation whose image in $GD_{X_0/S_0}(S)$ is $\G^\bullet$.
\end{proof}

\section{Thom--Whitney Resolutions}

In this section, we briefly review the Thom--Whitney resolution. These resolutions are acyclic resolutions of complexes that are adapted to preserve additional algebraic structures such as $\wedge$ and $[-,-]$ of a Gerstenhaber algebra; therefore, they have been employed in \cite{AlgebraicBTT2010} and \cite{ChanLeungMa2019} to study deformations. In their present form they first occur in \cite{AznarHodgeDeligne1987}. Their construction starts from a \emph{semicosimplicial} complex $V^\Delta$.
Recall that, denoting $\Delta_{\text{mon}}$ the category of sets $[n]=\{0,1,...,n\}$ for $n \geq 0$ 
with morphisms the order-preserving injective maps, a \emph{semicosimplicial} object in a category $\C$ is 
a covariant functor $\C^\Delta: \Delta_\text{mon} \to \C$.
Thus, a semicosimplicial object is a 
diagram
$$A_0 \ \substack{\longrightarrow\\[-0.9em] \longrightarrow} \ A_1 \
\substack{\longrightarrow\\[-0.9em] \longrightarrow \\[-0.9em] \longrightarrow} \ 
A_2 \
\substack{\longrightarrow\\[-0.9em] \longrightarrow \\[-0.9em] \longrightarrow
\\[-0.9em] \longrightarrow} \ \cdots  $$
where for each $n \geq 1$, we have $n + 1$ morphisms $\partial_{k,n}: A_{n - 1} \to A_n$
satisfying $\partial_{\ell,n+1}\partial_{k,n} = \partial_{k+1,n+1}\partial_{\ell,n}$.

\begin{ex}\label{ex cech1}
Let $X$ be a $k$-scheme, let
$\U = \{U_i\}$ be an affine cover of $X$, and let $\F$ be a sheaf of $k$-vector spaces. The \v{C}ech semicosimplicial sheaf $\F(\U)$ is the semicosimplicial sheaf
$$\prod_i\F(U_i) \ \substack{\longrightarrow\\[-0.9em] \longrightarrow} \ \prod_{i<j}\F(U_{ij}) \ 
\substack{\longrightarrow\\[-0.9em] \longrightarrow \\[-0.9em] \longrightarrow} \ \prod_{i<j<k}\F(U_{ijk}) \ 
\substack{\longrightarrow\\[-0.9em] \longrightarrow \\[-0.9em] \longrightarrow \\[-0.9em] \longrightarrow} \
\cdots $$
with the usual \v{C}ech maps, and with $\F(U) := j_*\F|_U$. Similarly, if $\G^\bullet$ is 
a sheaf of Gerstenhaber algebras on $X$, then $\G^\bullet(\U)$ is a semicosimplicial sheaf of Gerstenhaber algebras, i.e., each term $\G^\bullet(\U)_n = \prod_{i_0 < ... < i_n}\G^\bullet(U_{i_0...i_n})$ is a Gerstenhaber algebra, and the \v{C}ech maps are morphisms thereof.
\end{ex}

For each $n$, the differential forms on $\{t_0 + ... + t_n = 1\} \subset \bAA^{n + 1}$ form 
a differential graded commutative algebra 
$$(A_{PL})_n = k[t_0, ..., t_n, dt_0, ..., dt_n]/(1 - \sum t_i, \sum dt_i);$$
here, $\mathrm{deg}(t_i) = 0$ and $\mathrm{deg}(dt_i) = 1$. In particular, $k[t_0, ..., t_n, dt_0, ..., dt_n]$ is the \emph{graded} commutative polynomial ring with $dt_i \cdot dt_j = - dt_j \cdot dt_i$.
The inclusions $\bAA^{n} \to \bAA^{n + 1}$ of coordinate hyperplanes induce face maps
$\delta^{k,n}: (A_{PL})_n \to (A_{PL})_{n - 1}$, 
which turn $A_{PL}$ into a \emph{semisimplicial} dgca.
Given a semicosimplicial complex $V^\Delta$ of $k$-vector spaces---i.e., $V^\Delta([n]) =: (V_n^\bullet,d)$ is a complex---we use the maps $\delta^{k,n}$ together with the coface maps $\partial_{k,n}$ of $V^\Delta$
to define homomorphisms
\begin{align}
  \delta^{k,n} \otimes \text{Id}&: \quad (A_{PL})_n^i \otimes_k V_n^j \ \quad \to 
(A_{PL})_{n - 1}^i \otimes_k V_n^j \nonumber \\
\text{Id} \otimes \partial_{k,n}&: \quad (A_{PL})_{n-1}^i \otimes_k V_{n - 1}^j \to 
(A_{PL})_{n - 1}^i \otimes_k V_n^j \nonumber
\end{align}
of vector spaces. Following \cite{CoDGLADef2012}, the Thom--Whitney bicomplex 
has graded pieces 
$$C^{i,j}_\text{TW}(V^\Delta) = \{(x_n)_{n \in \NN} \in \prod_{n \in \NN} 
(A_{PL})_n^i \otimes V_n^j \ | \ (\delta^{k,n} \otimes \text{Id})(x_n) = (\text{Id} \otimes \partial_{k,n})(x_{n-1})\}$$
with differentials given by 
\begin{align}
 \delta_1: C_\text{TW}^{i,j}(V^\Delta) \to C^{i + 1,j}_\text{TW}(V^\Delta)&, \quad 
 (a_n \otimes v_n) \mapsto \qquad (da_n \otimes v_n), \nonumber \\
 \delta_2: C_\text{TW}^{i,j}(V^\Delta) \to C^{i,j + 1}_\text{TW}(V^\Delta)&, \quad
 (a_n \otimes v_n) \mapsto (-1)^i(a_n \otimes dv_n). \nonumber 
\end{align}
The Thom--Whitney complex 
$\text{Tot}_\text{TW}(V^\Delta) = \text{Tot}(C^{\bullet,\bullet}_\text{TW}(V^\Delta))$
is its total complex. The construction is functorial for homomorphisms 
of semicosimplicial dg vector spaces, and it is exact by
\cite[Lemma 2.4]{AznarHodgeDeligne1987}, where $U^\Delta \to V^\Delta \to W^\Delta$ is exact if it is exact on each $V_n^i$.

For later use, we prove a base change result on the Thom--Whitney construction. We say $V^\Delta$ is \emph{bounded} if $V_n^\bullet = 0$ for $n >> 0$.

\begin{lemma}\label{CTW-base-change}
 Let $R \to S$ be a finite type homomorphism of $k$-algebras, and let $V^\Delta$ be a bounded semicosimplicial complex of $R$-modules. Then the canonical map $C_\text{TW}^{i,j}(V^\Delta) \otimes_R S \to C_\text{TW}^{i,j}(V^\Delta \otimes_R S)$ is bijective.
\end{lemma}
\begin{proof}
 After factoring $R \to T \to S$---where $T = R[x_1,...,x_n]$ is a polynomial ring and $T \to S$ is surjective---it suffices to prove the statement for $R \to S$ either flat or surjective. For the flat case, first note that the product in the definition of $C_\text{TW}^{i,j}$ is actually a finite direct sum because $V^\Delta$ is bounded. Now $C_\text{TW}^{i,j}(V^\Delta)$ is the kernel of 
 $$\sum_{n,k} (\delta^{k,n} \otimes \text{Id} - \text{Id} \otimes \partial_{k,n}): \quad \bigoplus_n (A_{PL})_n^i \otimes V_n^j \to \bigoplus_{n} (A_{PL})_{n-1}^i \otimes V_n^j,$$
 so its formation commutes with flat tensor products. For the surjective case, let $I \subset R$ be the kernel. Because $C_\text{TW}^{i,j}$ is exact, we find a diagram
 \[
  \xymatrix{
   C^{i,j}_\text{TW}(V^\Delta) \otimes_R I \ar[r] \ar[d]^{p_I} & C^{i,j}_\text{TW}(V^\Delta) \ar[r]\ar@{=}[d]& C^{i,j}_\text{TW}(V^\Delta) \otimes_R S \ar[r]\ar[d]^{p_S} & 0 \\
   C^{i,j}_\text{TW}(V^\Delta \otimes_R I) \ar[r] & C^{i,j}_\text{TW}(V^\Delta) \ar[r] & C^{i,j}_\text{TW}(V^\Delta \otimes_R S) \ar[r] & 0 \\
  }
 \]
 with exact rows. Then $p_S$ is surjective, and a similar argument using a surjection $R^{\oplus m} \to I$ shows $p_I$ surjective as well. In particular, we find $p_S$ bijective.
\end{proof}

\subsubsection*{Thom--Whitney Resolutions on Schemes}

We globalize the Thom--Whitney construction to schemes. This has been done in \cite{Hinich1997}, but we recall it here for convenience.

\begin{constr}
 Let $X$ be a $k$-scheme and $\F^\Delta$ be a semicosimplicial complex of sheaves of $k$-vector spaces.
 Then applying the above construction on every open, we obtain a presheaf $C^{i,j}_\text{TW}(\F^\Delta)$, which is 
 a sheaf. Indeed, for a sheaf $\E$ of $k$-vector spaces, 
 $U \mapsto (A_{PL})^i_n \otimes_k \E(U)$ is a sheaf, products of sheaves are sheaves, and taking elements that satisfy an equation preserves the sheaf condition.
\end{constr}

As above, we say $\F^\Delta$ is \emph{bounded} if $\F^\bullet_n = 0$ for $n >> 0$. In case $\F^\Delta$ is a bounded semicosimplicial complex of quasi-coherent $\cO_X$-modules, also $C^{i,j}_\text{TW}(\F^\Delta)$ is a quasi-coherent $\cO_X$-module because the infinite product in the construction is actually finite.

\begin{lemma}\label{TW-base-change-compat}
 Let $c: Y \to X$ be a morphism of finite type of $k$-schemes, and let $\F^\Delta$ be a bounded semicosimplicial complex of quasi-coherent $\cO_X$-modules. Then the canonical map $c^*C^{i,j}_\text{TW}(\F^\Delta) \to C^{i,j}_\text{TW}(c^*\F^\Delta)$ is an isomorphism.
\end{lemma}
\begin{proof}
 The inverse image $c^*\F^\Delta$ is a bounded semicosimplicial complex of quasi-coherent $\cO_Y$-modules, so everything is quasi-coherent. It thus suffices to compare on some affine cover of $Y$, where it follows from Lemma~\ref{CTW-base-change}.
\end{proof}
\begin{ex}
 Let $X$ be a separated $k$-scheme of finite type, let $\U = \{U_i\}$ be a finite affine cover, and let $\F^\bullet$ be a complex of quasi-coherent $\cO_X$-modules. Then $\F^\bullet(\U)$ is a bounded semicosimplicial complex of quasi-coherent $\cO_X$-modules. Because the inclusions $j: U_{i_0...i_n} \to X$ are affine, we have $\F^\bullet(\U)\otimes \cO_Y = (\F^\bullet \otimes \cO_Y)(\U)$ for a thickening $c: Y \to X$, and hence the natural map 
 $$c^*C^{i,j}_\text{TW}(\F^\bullet(\U)) \to C^{i,j}_\text{TW}\left((c^*\F^\bullet)(c^{-1}(\U))\right)$$
 is an isomorphism.
\end{ex}

For a complex of sheaves $\F^\bullet$, the map
$$\F^j \to C^{0,j}_\text{TW}(\F^\bullet(\U)), \qquad f \mapsto (1 \otimes (f_{U|_{i_0...i_n}}))$$
induces a quasi-isomorphism $\F^\bullet \to \text{Tot}_\text{TW}(\F^\bullet(\U))$; this is shown by postcomposing with 
the integration quasi-isomorphism to the \v{C}ech complex
$\check\C^\bullet(\U,\F^\bullet)$. 

\begin{lemma}\label{tw-acyclic}
 Let $X$ be a separated $k$-scheme of finite type, let $\U = \{U_i\}$ be a finite affine cover, and let $\F^\bullet$ be a complex of quasi-coherent $\cO_X$-modules. Then we have 
 $H^\ell(X,C^{i,j}_\mathrm{TW}(\F^\bullet(\U))) = 0$ for $\ell \geq 1$.
\end{lemma}
\begin{proof}
 We compute the cohomology via the \v Cech complex 
 $\check \C^\bullet (\U,C^{i,j}_\mathrm{TW}(\F^\bullet)) $; our proof roughly follows \cite[Lemma 3.27]{ChanLeungMa2019}. Note that for an open $V \subset X$, we have 
 \begin{align}
  C^{i,j}_\mathrm{TW}(\F^\bullet(\U))(V) = 
  &\{(f_{i_0...i_n}) \in \prod_{i_0 < ... < i_n} (A_{PL})^i_n \otimes \F^j(U_{i_0...i_n} \cap V) \ | \nonumber \\
  &\ \forall k \leq n: (\delta_{k,n} \otimes \mathrm{Id})(f_{i_0...i_n}) = f_{i_0...\hat i_k ... i_n}|_{U_{i_0...i_n} \cap V} \} 
  \label{tw-eq}
 \end{align}
 where the product runs over all (finitely many) ordered tuples $(i_0,...,i_n)$. Given an element $(f_{\alpha_0...\alpha_\ell;i_0...i_n}) \in \check \C^\ell (\U,C^{i,j}_\mathrm{TW}(\F^\bullet))$, we find 
 $$(\check d(f_{\alpha_0...\alpha_\ell;i_0...i_n}))_{\alpha_0...\alpha_{\ell + 1};i_0...i_n} = \sum_{k = 0}^{\ell + 1}(-1)^kf_{\alpha_0...\hat{\alpha_k}...\alpha_{\ell + 1};i_0...i_n}|_{U_{\alpha_0...\alpha_{\ell + 1}}\cap U_{i_0...i_n}}$$
 for the \v Cech differential, so we can work in $(A_{PL})_n^i \otimes\check\C^\bullet
 (\U,\F^j(U_{i_0...i_n}))$ (notation from Example~\ref{ex cech1}) for each tuple $(i_0,...i_n)$ separately. Because $(A_{PL})_n^i \otimes\check\C^\bullet
 (\U,\F^j(U_{i_0...i_n}))$ does not have cohomology in degree $\not= 0$, we can find elements 
 $$F_{\alpha_0...\alpha_{\ell - 1};i_0...i_n} \in (A_{PL})^i_n \otimes \F^j(U_{\alpha_0...\alpha_{\ell - 1}} \cap U_{i_0...i_n})$$
 with $\check d(F_\bullet) = (f_\bullet)$, but we might have $(F_\bullet) \notin \check \C^{\ell - 1} (\U,C^{i,j}_\mathrm{TW}(\F^\bullet(\U)))$ since $(F_\bullet)$ might not satisfy equation \eqref{tw-eq}.
 
 We correct $(F_\bullet)$ to an element $(v_\bullet) \in \check \C^{\ell - 1} (\U,C^{i,j}_\mathrm{TW}(\F^\bullet(\U)))$ by induction on $n$. For $n < i$ we have $(A_{PL})^i_n = 0$, so we set $v_{\bullet;i_0...i_n} = 0$. For $n = i$ we take $v_{\bullet;i_0...i_n} = F_{\bullet;i_0...i_n}$ which satisfies equation \eqref{tw-eq} (when we plug it in on the left). For the induction step, let $K_{n + 1}$ be the kernel of 
 $\check\C^{\ell - 1}
 (\U,\F^j(U_{i_0...i_{n+1}})) \to \check\C^\ell
 (\U,\F^j(U_{i_0...i_{n + 1}}))$, and note that 
 $\sum \delta_{k,n + 1}: (A_{PL})^i_{n + 1} \to \bigoplus_{k = 0}^{n + 1}(A_{PL})^i_n$ is surjective by \cite[Lemma 8.3]{GriffithsMorgan1981} (or \cite[Lemma 3.5]{ChanLeungMa2019}).
 We assemble these spaces into a diagram
 \[
  \xymatrix{
   (A_{PL})^i_{n + 1} \otimes K_{n + 1} \ar[r] \ar@{^{(}->}[d] & \bigoplus_{k = 0}^{n + 1}(A_{PL})^i_n \otimes K_{n + 1} \ar@{^{(}->}[d] \\
   (A_{PL})^i_{n + 1} \otimes \check\C^{\ell - 1}
 (\U,\F^j(U_{i_0...i_{n+1}})) \ar[r]^{\Delta^{\ell - 1}} \ar[d]^{\check d} & \bigoplus_{k = 0}^{n + 1}(A_{PL})^i_n \otimes \check\C^{\ell - 1}
 (\U,\F^j(U_{i_0...i_{n+1}})) \ar[d] \\
 (A_{PL})^i_{n + 1} \otimes \check\C^{\ell}
 (\U,\F^j(U_{i_0...i_{n+1}})) \ar[r] &
 \bigoplus_{k = 0}^{n + 1}(A_{PL})^i_n \otimes \check\C^{\ell}
 (\U,\F^j(U_{i_0...i_{n+1}})) \\
  }
 \]
 with exact columns and surjective rows. Assuming $(v_\bullet)$ to be constructed up to order $n$, we need to find 
 $$(v_{\bullet;i_0...i_{n + 1}}) \in (A_{PL})^i_{n + 1} \otimes \check\C^{\ell - 1}
 (\U,\F^j(U_{i_0...i_{n+1}})) $$ 
 such that 
 $\check d(v_{\bullet;i_0...i_{n + 1}}) = (f_{\bullet;i_0...i_{n + 1}})$ and---to satisfy equation \eqref{tw-eq}---we have 
 $\Delta^{\ell - 1}(v_{\bullet;i_0...i_{n + 1}}) = \hat v_{i_0...i_{n + 1}}$ where 
 $\hat v_{i_0...i_{n + 1}}$ is constructed from $(v_\bullet)$ (up to order $n$) by the right hand side of equation \eqref{tw-eq}.
 Given $(F_{\bullet;i_0...i_{n + 1}})$ this is an easy diagram chase.
\end{proof}

\begin{cor}\label{global-sec-base-change-compat}
 Let $A' \to A$ be a morphism in $\mathbf{Art}_Q$, let $f': X' \to S'$ be separated, let $\U'$ be a finite affine cover of $X'$, and let $X = X' \times_{S'} S$. Let $\F^\bullet$ be a complex of quasi-coherent $\cO_{X'}$-modules which are flat over $A'$, and let $\C := C^{i,j}_\text{TW}(\F^\bullet(\U'))$. Then the canonical map $H^0(X',\C) \otimes_{A'} A \to H^0(X,\C \otimes_{A'} A)$ is bijective.
\end{cor}
\begin{proof}
 We have a factorization $A' \to A'[x_1,...,x_n]/(x_1^{m_1},...,x_n^{m_n}) \to A$ in $\mathbf{Art}_Q$ with the second map surjective, so it suffices to prove the statement for $A' \to A$ either flat or a small extension. The flat case is by flat base change. In the small extension case, let $I \subset A'$ be the ideal. Then we have $I \cdot \C \cong (\C \otimes_{A'} k) \otimes_k I$ due to flatness, so Proposition~\ref{tw-acyclic} shows $H^1(X',I \cdot \C) = 0$; thus, the result follows from the long exact sequence in sheaf cohomology.
\end{proof}

\section{The Thom--Whitney Gerstenhaber Algebra}

We perform a Thom--Whitney resolution of the Gerstenhaber algebras $G^\bullet_{V_\alpha/S}$ of polyvector fields and glue them canonically to a global sheaf of (bigraded) Gerstenhaber algebras $\mathrm{PV}_{X_0/S}$. This sheaf depends on $X_0$ and the base $S$, but not on a deformation $f: X \to S$, thus the notation ``$X_0/S$''. The gluing is inspired by the gluing construction in \cite[3.3]{ChanLeungMa2019}.

Let $\G^\bullet$ be a Gerstenhaber $k$-algebra on a (separated) finite type $k$-scheme $X$, and let $\U$ be an open affine cover. To perform a Thom--Whitney resolution, we turn $\G^\bullet$ into a complex of coherent sheaves by endowing it with the differential $0$. We set
$\textrm{TW}^{p,q}(\G^\bullet) := C^{q,p}_\textrm{TW}(\G^\bullet(\U))$; the switch of indices is on purpose to fit the conventions of \cite{ChanLeungMa2019}. The operations
$$(a \otimes v) \wedge (b \otimes w) := (-1)^{|b||v|}(a \wedge b) \otimes (v \wedge w)$$
 and 
 $$[a \otimes v, b \otimes w] := (-1)^{(|v| + 1)|b|}(a\wedge b) \otimes [v,w]$$
 turn it into a \emph{bigraded Gerstenhaber algebra} $\TW^{\bullet,\bullet}(\G^\bullet)$ (or $\TW^{\bullet,\bullet}$ for brevity), i.e., we have 
 $$\TW^{p,q} \wedge \TW^{i,j} \subset \TW^{p + i, q + j} \quad \mathrm{and}\quad [\TW^{p,q},\TW^{i,j}] \subset \TW^{p + i + 1, q + j}$$ 
 as well as the usual relations with respect to the total degree $p + q$. The differential $d: \TW^{p,q} \to \TW^{p,q+1}, a \otimes v \mapsto da \otimes v$, satisfies the relations
 \begin{equation}\label{GD}
  d(x \wedge y) = dx \wedge y + (-1)^{|x|} x \wedge dy, \quad d[x,y] = [dx,y] + (-1)^{|x| + 1}[x,dy]
 \end{equation}
 and $d^2 = 0$ whereas the other differential $\TW^{p,q} \to \TW^{p + 1,q}$ of this Thom--Whitney bicomplex is $0$. We say that its total complex $\TW^\bullet(\G^\bullet) := \mathrm{Tot}_\mathrm{TW}(\G^\bullet(\U))$ is a \emph{(strongly) differential Gerstenhaber algebra}. The quasi-isomorphism $(\G^\bullet,0) \to (\TW^\bullet(\G^\bullet),d)$ is a functorial acyclic resolution of differential Gerstenhaber algebras. In particular, we can recover $\G^\bullet$ as the cohomology sheaves $\cH^\bullet(\mathrm{TW}^\bullet(\G^\bullet))$ which form a Gerstenhaber algebra for every differential Gerstenhaber algebra. Sometimes we write $\cH^\bullet(\TW^{\bullet,\bullet})$ for the cohomology; then taking the total complex is implicit---cohomologies are, for us, always singly graded. However, in bidegree $(p,q)$ with $q > 0$, the cohomology of $d: \TW^{p,q} \to \TW^{p, q + 1}$ is $0$ anyway.
 
 \begin{rem}
  Do not confuse our notion of differential Gerstenhaber algebra with the one of e.g.~\cite{Kosmann1995} which is more closely related to Batalin--Vilkovisky algebras.
 \end{rem}

 \begin{ex}\label{local-TW-ex}
  Let $f: X \to S$ be a log smooth and saturated deformation of $f_0: X_0 \to S_0$. Then we have a resolution $G^\bullet_{X/S} \to \TW^\bullet(G^\bullet_{X/S})$. If we deform further to $f': X' \to S'$, then applying $\TW(-)$ to the exact sequence \eqref{rel-def-sequence} we obtain the exact sequence
  $$ 0 \to \I_{X'/X} \cdot \TW^\bullet(G^\bullet_{X'/S'}) \to \TW^\bullet(G^\bullet_{X'/S'}) \to \TW^\bullet(G^\bullet_{X/S}) \to 0$$
  because $\TW^\bullet(G^\bullet_{X'/S'})$ is flat over $S'$ and compatible with base change. 
  
  If $\mathrm{exp}_\theta: G^\bullet_{X'/S'} \to G^\bullet_{X'/S'}$ is a gauge transform (relative to $f: X \to S$), then the induced automorphism $\TW(\mathrm{exp}_\theta)$ is the gauge transform defined by $(1 \otimes (\theta|_{U_{i_0...i_n}})) \in \I_{X'/X}\cdot \TW^{-1,0}(G^\bullet_{X'/S'})$ via the formula in \eqref{gauge-def-formula}.
 \end{ex}
 
 $(-1)$-injectivity is preserved by the Thom--Whitney construction. This is important for Lemma~\ref{TWG-first-order-def} below.
 
 \begin{lemma}\label{injective-G}
  Let $\G^\bullet$ be a $(-1)$-injective Gerstenhaber algebra. Then for $0 \not= \theta \in \TW^{-1,j}(\G^\bullet)$ the map 
  $[\theta,-]: \TW^{0,0}(\G^\bullet) \to \TW^{0,j}(\G^\bullet)$ is \emph{not} the zero map.
 \end{lemma}
 \begin{proof}
  We work over an arbitrary open $V \subset X$. Following the description of $C^{i,j}_\mathrm{TW}$ in Lemma~\ref{tw-acyclic}, the element $\theta \in \TW^{-1,j}(\G^\bullet)$ is given by a family $(\theta_{i_0...i_n})$. After denoting $\{t_\mu\}$ some basis of $(A_{PL})^j_n$, we can decompose 
  $\theta_{i_0...i_n} = \sum_\mu a_\mu t_\mu \otimes \theta_{i_0...i_n}^\mu$ with 
  $\theta^\mu_{i_0...i_n} \in \G^{-1}(U_{i_0...i_n} \cap V)$. For \emph{some non-zero} $\theta^\mu_{i_0...i_n}$ we find
  $f \in \cO_X(U_{i_0...i_n} \cap V)$ such that $[\theta^\mu_{i_0...i_n},f] \not= 0$, hence $[\theta,(1 \otimes f)] \not= 0$ for the induced element $(1 \otimes f) \in \TW^{0,0}(\G^\bullet)(U_{i_0...i_n} \cap V)$.
 \end{proof}

 We apply this construction to the Gerstenhaber algebras $G^\bullet_{V_\alpha/S}$ and obtain
differential Gerstenhaber algebras $\TW^\bullet(G^\bullet_{V_\alpha/S})$ as well as isomorphisms $\TW(\phi_{\alpha\beta})$ on 
overlaps. Every cocycle 
$$\TW(\phi_{\gamma\alpha}) \circ \TW(\phi_{\beta\gamma}) \circ\TW(\phi_{\alpha\beta})$$
is a gauge transform by an element in $\I \cdot \TW^{-1,0}(G^\bullet_{V_\alpha/S}|_{V_\alpha \cap V_\beta \cap V_\gamma})$, where $\I$ is the kernel of $\cO_X \to \cO_{X_0}$. When gluing them as differential Gerstenhaber algebras, this induces a Gerstenhaber deformation (Definition~\ref{gerstenhaber-defo-defn}) by taking cohomology. As a step towards that, our goal is to glue them as a \emph{bigraded Gerstenhaber algebra}, i.e., \emph{without} differential (and keeping the bigrading).

\begin{defn}\label{TWG-def}
 A \emph{Thom--Whitney--Gerstenhaber (TWG) deformation} is a bigraded Gerstenhaber algebra $\T^{\bullet,\bullet}$ with a morphism $\T^{\bullet,\bullet} \to \TW^{\bullet,\bullet}(G^\bullet_{X_0/S_0})$ (of bigraded Gerstenhaber algebras) and isomorphisms 
 $\chi_\alpha: \T^{\bullet,\bullet}|_{V_\alpha} \cong \TW^{\bullet,\bullet}(G^\bullet_{V_\alpha/S})$ compatible with the maps to 
 $\TW^{\bullet,\bullet}(G^\bullet_{X_0/S_0})$ such that the cocycle 
 $\phi_{\alpha\beta} \circ \chi_\alpha \circ \chi_\beta^{-1}$ is a gauge transform defined by an element in 
 $\I \cdot \TW^{-1,0}(G^\bullet_{V_\beta/S}|_{V_\alpha \cap V_\beta})$. Isomorphisms are analogous to Gerstenhaber deformations.
\end{defn}

 In general, gauge transforms are not compatible with the differentials, so there is no canonical differential on $\T^{\bullet,\bullet}$ coming out of the data. The canonical example of a TWG deformation is the following.

\begin{ex}
 Let $\G^\bullet$ be a Gerstenhaber deformation. Then 
 $\TW^{\bullet,\bullet}(\G^\bullet)$ is a TWG deformation upon forgetting the differential.
\end{ex}

Given a morphism $S \to S'$ in $\mathbf{Art}_Q$ and a TWG deformation $\T^{\bullet,\bullet}$ on $S'$, Lemma~\ref{TW-base-change-compat} shows that $\T^{\bullet,\bullet}\otimes_{A'} A$ is a TWG deformation as well (where $A = \cO(S)$).
Moreover, TWG deformations have a deformation theory just like classical flat deformations:

\begin{lemma}\label{TWG-first-order-def}
 Let $0 \to I \to A' \to A \to 0$ be a small extension in $\mathbf{Art}_Q$ and let $\T^{\bullet,\bullet}$ be a TWG deformation on $S = \Spec A$. Then:
 \begin{itemize}
  \item Given a lifting $\T'^{\bullet,\bullet}$ on $S'$, the relative automorphisms are in 
  $$H^0(X_0,\TW^{-1,0}(G^\bullet_{X_0/S_0})) \otimes_k I.$$
  \item Given a lifting $\T'^{\bullet,\bullet}$ on $S'$, the isomorphism classes of liftings are in 
  $$H^1(X_0,\TW^{-1,0}(G^\bullet_{X_0/S_0})) \otimes_k I.$$
  \item The obstructions to the existence of a lifting are in 
  $$H^2(X_0,\TW^{-1,0}(G^\bullet_{X_0/S_0})) \otimes_k I.$$
 \end{itemize}
\end{lemma}
Since $\TW^{-1,0}(G^\bullet_{X_0/S_0})$ is acyclic by Lemma~\ref{tw-acyclic}, there is indeed a unique lifting up to isomorphism. This means for every $S$, there is up to isomorphism a unique TWG deformation 
 $\PV_{X_0/S}$ (depending on $S$ and the morphism $f_0: X_0 \to S_0$, but no further data) which we call the \emph{Thom--Whitney Gerstenhaber algebra}. In fact, $\PV_{X_0/S}$ has many automorphisms, so there is no canonical choice. Thus we fix once and for all one $\PV_{X_0/S}$ for every $S \in \mathbf{Art}_Q$. The notation $\PV$ is taken from \cite{ChanLeungMa2019}; presumably, $\PV$ stands for ``polyvector fields'' since it is obtained by resolving the Gerstenhaber algebra of polyvector fields.
\begin{proof}[Proof of the Lemma]
  Given a lifting $\T'^{\bullet,\bullet}$, we consider the choice of a morphism 
 $\T'^{\bullet,\bullet} \to \T^{\bullet,\bullet}$ that is the given one on $\TW^{\bullet,\bullet}(G^\bullet_{V_\alpha/S'}) \to \TW^{\bullet,\bullet}(G^\bullet_{V_\alpha/S})$ as part of the datum.
 Since the gauge transforms $\phi_{\alpha\beta} \circ \chi_\alpha \circ \chi_\beta^{-1}$ are $A'$-linear, the TWG deformation $\T'^{\bullet,\bullet}$ consists of sheaves of $A'$-modules. Thus the sequence 
 $$0 \to I \cdot \T'^{\bullet,\bullet} \to \T'^{\bullet,\bullet} \to \T^{\bullet,\bullet} \to 0$$
 makes sense and it is exact because it is locally the sequence of Example~\ref{local-TW-ex}. Every element $\theta \in I \cdot \T'^{-1,0}$ induces a gauge transform $\mathrm{exp}_\theta$ by the formula in Section~\ref{def-gerstenhaber-algebra-sec}. It is an automorphism of $\T'^{\bullet,\bullet}$ in the sense of Definition~\ref{TWG-def} because it is a gauge transform on $\TW^{\bullet,\bullet}(G^\bullet_{V_\alpha/S'})$ as well, so we obtain a sheaf map 
 $$\mathrm{exp}_{\T'/\T}: I \cdot \T'^{-1,0} \to \A ut_{\T'/\T}, \quad \theta \mapsto (\xi \mapsto \xi + [\theta,\xi])$$
 into the sheaf of lifting automorphisms. It is injective due to Lemma~\ref{injective-G}. By the very definition of an automorphism of TWG deformations, every $\varphi \in \A ut_{\T'/\T}$ is (locally) induced by some $\theta \in \m \cdot \T'^{-1,0}$ where $\m \subset A'$ is the maximal ideal. Because the induced automorphism on $\T^{\bullet,\bullet}$ is the identity, we have indeed $\theta \in I \cdot \T'^{-1,0}$, so $\mathrm{exp}_{\T'/\T}$ is an isomorphism.
 Finally we have $\TW^{-1,0}(G^\bullet_{X_0/S_0}) \otimes_k I \cong I \cdot \T'^{-1,0}$, so the result follows by the standard methods that are developed for smooth deformations in \cite[III]{SGAI}.
\end{proof}

\begin{cor}
 Let $I_S \subset A = \cO(S)$ be the maximal ideal. Then every automorphism of a TWG deformation $\T$ is a gauge transform, i.e., $\mathrm{exp}_\T: I_S \cdot \T^{-1,0} \to \A ut_\T$ is an isomorphism (once we endow $I_S \cdot \T^{-1,0}$ with the Baker--Campbell--Hausdorff product).  
\end{cor}
\begin{proof}
 This follows by induction on the length of $A$, breaking the extension into small extensions.
\end{proof}

\begin{rem}
 If $S \to S'$ is a map in $\mathbf{Art}_Q$ corresponding to a homomorphism $A' \to A$ of Artin rings, then $\PV_{X_0/S'} \otimes_{A'} A$ is a TWG deformation on $S$. In particular, there is an isomorphism to $\PV_{X_0/S}$. Moreover, we have an induced homomorphism $\PV_{X_0/S'} \to \PV_{X_0/S'} \otimes_{A'} A$. However, there is \emph{no} canonical 
 homomorphism $\PV_{X_0/S'} \to \PV_{X_0/S}$ since there is no preferred isomorphism $\PV_{X_0/S'} \otimes_{A'} A \cong \PV_{X_0/S}$. This is because $\PV_{X_0/S}$ has many automorphisms. We see that $S \mapsto \PV_{X_0/S}$ is not functorial in a strict sense.
\end{rem}

\section{Differentials on TWG Deformations}

After constructing a unique TWG deformation $\PV_{X_0/S}$, we now study differentials $d: \PV_{X_0/S}^{p,q} \to \PV_{X_0/S}^{p,q+1}$ on it. Recall from 
e.g.~\cite[Lemma 2.5]{ChanLeungMa2019} that if $\theta \in \I_{X/X_0} \cdot \TW^{-1,0}(G^\bullet_{X/S})$, then 
\begin{equation}\label{gauge-formula}
 \mathrm{exp}_\theta \circ (d + [\xi,-]) \circ \mathrm{exp}_{-\theta} = d + [\mathrm{exp}_\theta(\xi) - T([\theta,-])(d\theta),-]
\end{equation}
where $T([\theta,-])$ means to plug in the operator 
$[\theta,-]$ into the power series expansion of $T(x) = \frac{\mathrm{exp}(x) - 1}{x}$. This means under a gauge transform, the differential $d$ is transformed into something of the form $d + [\eta,-]$ for $\eta \in \I_{X/X_0}\cdot \TW^{-1,1}(G^\bullet_{X/S})$. Hence every differential on $\PV_{X_0/S}$ should be locally of this form. More formally we define:

\begin{defn}
 Let $\T^{\bullet,\bullet}$ be a TWG deformation. Then a 
 \emph{predifferential} is a map $d: \T^{p,q} \to \T^{p,q + 1}$ that satisfies \eqref{GD}, that is compatible with 
 the differential on $\TW^{\bullet,\bullet}(G^\bullet_{X_0/S_0})$, and such that 
 $$d|_{V_\alpha} - \chi_\alpha^{-1} \circ d_\alpha \circ \chi_\alpha = [\eta_\alpha,-]$$ for some $\eta_\alpha \in \I \cdot \T^{-1,1}|_{V_\alpha}$ where $d_\alpha$ is the differential on $\TW^{\bullet,\bullet}(G^\bullet_{V_\alpha/S})$.
 It is a \emph{differential} if $d^2 = 0$. We denote the set of predifferentials by $\mathrm{PDiff}(\T^{\bullet,\bullet})$ and the set of differentials by $\mathrm{Diff}(\T^{\bullet,\bullet})$. Two predifferentials $d_1,d_2$ are \emph{gauge equivalent}, if there is an automorphism $\psi: \T^{\bullet,\bullet} \to \T^{\bullet,\bullet}$ of TWG deformations such that 
 $\psi \circ d_1 = d_2 \circ \psi$.
\end{defn}
\begin{ex}\label{G-to-diff}
 Let $\G^\bullet$ be a Gerstenhaber deformation. Then 
 the differential of $\TW^{\bullet,\bullet}(\G^\bullet)$ is a differential in the above sense. Since $\TW^{\bullet,\bullet}(\G^\bullet) \cong \PV_{X_0/S}$, it induces a differential on $\PV_{X_0/S}$ whose gauge equivalence class does not depend on the chosen isomorphism.
\end{ex}

Every restriction $\PV_{X_0/S'} \to \PV_{X_0/S}$ induces a restriction $\mathrm{Diff}(\PV_{X_0/S'}) \to \mathrm{Diff}(\PV_{X_0/S})$ on differentials. On gauge equivalence classes, this restriction is independent of the chosen map 
$\PV_{X_0/S'} \to \PV_{X_0/S}$, so we obtain a 
functor 
$$TD_{X_0/S_0}: \mathbf{Art}_Q \to \mathbf{Set}$$
of Artin rings. To construct an inverse of the natural transformation 
$\mathbf{tw}: GD_{X_0/S_0} \Rightarrow TD_{X_0/S_0}$ given by Example~\ref{G-to-diff}, we need a lemma.

\begin{lemma}\label{local-gauge-unique}
 Let $\eta_\alpha \in I_S \cdot \TW^{-1,1}(G^\bullet_{V_\alpha/S})$ be such that $d_\alpha + [\eta_\alpha,-]$ is a differential. Then the two differentials $d_\alpha$ and $d_\alpha + [\eta_\alpha,-]$ are gauge equivalent.
\end{lemma}
\begin{proof}
 First assume $S \to S'$ is a small extension, set $\T' := \TW^{\bullet,\bullet}(G^\bullet_{V_\alpha/S'})$ etc. and consider $\eta_\alpha \in I_{S'}\cdot \T'^{-1,1}$ such that $\eta_\alpha|_S = 0$. Then the condition that $d_\alpha + [\eta_\alpha,-]$ is a differential (has square zero) becomes $d\eta_\alpha = 0$ (cf.~Maurer--Cartan equation), and for $\theta \in I \cdot \T'^{-1,0}$ formula~\eqref{gauge-formula} simplifies to 
 $ \mathrm{exp}_\theta \circ d \circ \mathrm{exp}_{-\theta} = d + [-d\theta,-]$. Because $V_\alpha$ is affine, we have 
 $H^1(V_\alpha, G^{-1}_{V_\alpha/S'}) = 0$, so
 we can find a gauge transform in $\A ut_{\T'/\T}$ with 
 $\mathrm{exp}_\theta \circ d_\alpha = (d + [\eta_\alpha,-]) \circ \mathrm{exp}_\theta$. 
 
 Now the proof is by induction on the length of $A = \cO(S)$. If $\eta_\alpha \in I_{S'}\cdot \T'^{-1,1}$ such that $d_\alpha + [\eta_\alpha,-]$ is a differential, then the restrictions to $\T$ are gauge equivalent. The gauge transform can be lifted to $\T'$, giving a gauge equivalence of $d_\alpha + [\eta_\alpha,-]$ and $d_\alpha + [\eta_\alpha',-]$ for some $\eta_\alpha'$ with $\eta_\alpha'|_S = 0$, the latter being gauge equivalent to $d_\alpha$ by the above argument.
\end{proof}

Given a differential $d \in \mathrm{Diff}(\PV_{X_0/S})$, the isomorphism 
$$\chi_\alpha: (\PV_{X_0/S},d)|_{V_\alpha} \cong (\TW^{\bullet,\bullet}(G^\bullet_{V_\alpha/S}),d_\alpha + [\chi_\alpha(\eta_\alpha),-])$$ is compatible with differentials. By Lemma~\ref{local-gauge-unique} we can further compose with a gauge transform to $(\TW^{\bullet,\bullet}(G^\bullet_{V_\alpha/S}),d_\alpha)$. Thus taking cohomology yields a Gerstenhaber deformation $\cH^\bullet(\PV_{X_0/S},d)$ whose isomorphism class does not depend on the chosen gauge transform (namely, every gauge transform that leaves the differential invariant descends to a gauge transform on cohomology). Likewise, if $d_1$ and $d_2$ are gauge equivalent differentials, the induced Gerstenhaber deformations are isomorphic. We find a natural transformation
$$\mathbf{h}: TD_{X_0/S_0} \Rightarrow GD_{X_0/S_0}$$
which is inverse to the natural transformation $\mathbf{tw}$ constructed above.

\begin{lemma}
 The functors $GD_{X_0/S_0}$ and $TD_{X_0/S_0}$ are naturally equivalent.
\end{lemma}
\begin{proof}
 Given a Gerstenhaber deformation $\G^\bullet$, we have a resolution $\G^\bullet \to \TW^\bullet(\G^\bullet)$, so $\mathbf{h} \circ \mathbf{tw} = \mathrm{Id}$. It thus suffices to prove that $\mathbf{h}$ is injective. Let $d_1,d_2$ be two differentials and let $\psi: \cH^\bullet(\PV_{X_0/S},d_1) \cong \cH^\bullet(\PV_{X_0/S},d_2)$ be an isomorphism. Fix the structure of Gerstenhaber deformation on the cohomologies by choosing isomorphisms $\tilde\chi_{\alpha,i}: (\PV_{X_0/S},d_i)|_{V_\alpha} \cong (\TW^{\bullet,\bullet}(G^\bullet_{V_\alpha/S}),d_\alpha)$ as in the construction of $\mathbf{h}$. The map $\psi$ induces a gauge transform $\mathrm{exp}_{\theta_\alpha}: G^\bullet_{V_\alpha/S} \to G^\bullet_{V_\alpha/S}$ which has a unique lift to a gauge transform $\mathrm{exp}_{\tilde\theta_\alpha}: \TW^{\bullet,\bullet}(G^\bullet_{V_\alpha/S}) \to \TW^{\bullet,\bullet}(G^\bullet_{V_\alpha/S})$ since $\TW^{-1,\bullet}(G^\bullet_{V_\alpha/S})$ is a resolution of $G^{-1}_{V_\alpha/S}$. Using the $\tilde\chi_{\alpha,i}$, we obtain an isomorphism $(\PV_{X_0/S},d_1)|_{V_\alpha} \to (\PV_{X_0/S},d_2)|_{V_\alpha}$ inducing $\psi$ in cohomology. Again by uniqueness they glue to a global isomorphism, so $d_1,d_2$ are gauge equivalent.
\end{proof}

\section{Maurer--Cartan Elements}

We construct a $k\llbracket Q\rrbracket$-linear pdgla $L^\bullet_{X_0/S_0}$ which controls $TD_{X_0/S_0}$ and thus log smooth deformations. To this end, we relate differentials on $\PV_{X_0/S}$ to an appropriate Maurer--Cartan equation.

The predifferentials $\mathrm{PDiff}(\PV_{X_0/S})$---considered as a sheaf on $X_0$---form an $I_S \cdot \PV_{X_0/S}^{-1,1}$-torsor for its additive group structure. Indeed, for $\eta \in I_S \cdot \PV_{X_0/S}^{-1,1}$ and $d \in \mathrm{PDiff}(\PV_{X_0/S})$, also $d + [\eta,-]$ is a predifferential. Every predifferential is (locally) of this form and Lemma~\ref{injective-G} shows that for $\eta \not= \eta'$ we have $d + [\eta,-] \not= d + [\eta',-]$. Using Lemma~\ref{tw-acyclic} inductively over small extensions, we find 
$H^\ell(X_0,\PV^{i,j}_{X_0/S}) = 0$ for $\ell \geq 1$. This suggests $H^1(X_0,I_S \cdot \PV^{-1,1}_{X_0/S}) = 0$ (which does not follow immediately), so there would be always a predifferential. In fact, more is true:

\begin{lemma}\label{prediff-lift}
 Let $S \to S'$ be a small extension and let $d \in \mathrm{PDiff}(\PV_{X_0/S})$. Choose a restriction $\PV_{X_0/S'} \to \PV_{X_0/S}$ (which is unique up to automorphisms of $\PV_{X_0/S}$). Then there is a predifferential $d' \in \mathrm{PDiff}(\PV_{X_0/S'})$ with $d'|_S = d$.
\end{lemma}
\begin{proof}
 The sheaf of predifferentials on $\PV_{X_0/S'}$ that restrict to $d$ over $S$ is an $I\cdot \PV_{X_0/S'}^{-1,1}$-torsor. Since $H^1(X_0,I\cdot \PV_{X_0/S'}^{-1,1}) \cong H^1(X_0,\TW^{-1,1}(G^\bullet_{X_0/S_0}) \otimes I) = 0$, it has a global section.
\end{proof}
\begin{cor}\label{trivial-torsor}
 $\mathrm{PDiff}(\PV_{X_0/S})$ is a trivial $I_S \cdot \PV_{X_0/S}^{-1,1}$-torsor. Every predifferential $d$ induces a bijection $I_S \cdot \PV_{X_0/S}^{-1,1} \to \mathrm{PDiff}(\PV_{X_0/S}), \eta \mapsto d_\eta := d + [\eta,-]$.
\end{cor}
\begin{rem}
 The predifferential that we find here corresponds to the operator $\bar\partial_\alpha + [\mathfrak{d}_\alpha,\cdot]$ in \cite[Thm.~3.34]{ChanLeungMa2019}.
\end{rem}

Given a differential $d \in \mathrm{Diff}(\PV_{X_0/S})$ and $\eta \in I_S \cdot \PV_{X_0/S}^{-1,1}$, we find 
$$(d + [\eta,-])^2(\kappa) = [d\eta + \frac{1}{2}[\eta,\eta],\kappa] =: [\ell,\kappa] \ .$$
The element $\ell \in \PV_{X_0/S}^{-1,2}$ is unique with that property by Lemma~\ref{injective-G}, so indeed for every predifferential $d$ there is a unique $\ell = \ell(d) \in \PV_{X_0/S}^{-1,2}$ with $d^2 = [\ell,-]$ (because locally we can compare it to a differential). Now if $d$ is a predifferential, then 
$$(d + [\eta,-])^2 = [d\eta + \frac{1}{2}[\eta,\eta] + \ell(d),-]$$
so $d + [\eta,-]$ is a differential if and only if $d\eta + \frac{1}{2}[\eta,\eta] + \ell = 0$. This is the \emph{Maurer--Cartan equation}.

\begin{rem}
 Our Maurer--Cartan equation corresponds to the \emph{classical Maurer--Cartan equation} in \cite[Defn.~5.10]{ChanLeungMa2019}.
\end{rem}

Our next goal is to construct the $k\llbracket Q\rrbracket$-pdgla $L^\bullet_{X_0/S_0}$. Let $\A := k\llbracket Q\rrbracket$, let $\m_Q \subset \A$ be the maximal ideal, let $\A_k := \A/\m_Q^{k + 1}$, and let $S_k := \Spec (Q \to \A_k)$. Then $S_k \to S_{k + 1}$ is a small extension, so after choosing restrictions $\PV_{X_0/S_{k + 1}} \to \PV_{X_0/S_k}$ we choose inductively compatible predifferentials $d_k$ on $\PV_{X_0/S_k}$ by Lemma~\ref{prediff-lift}. For $S \in \mathbf{Art}_Q$, there is at most one morphism $S \to S_k$, and there is a minimal $k$ such that $\mathrm{Hom}(S,S_k) \not= \emptyset$. For this minimal $k$, we choose a restriction $\PV_{X_0/S_k} \to \PV_{X_0/S}$ and we endow $\PV_{X_0/S}$ with the restricted predifferential $d_k|_S$. The construction yields an $\cO(S)$-pdgla 
$$L^\bullet(S) := (\Gamma(X_0,\PV^{-1,\bullet}_{X_0/S}),\ [-,-],\ d_k|_S,\ \ell(d_k|_S))$$
which is functorial for morphisms $S \to S_\ell$ (where $\ell \geq k$). In particular, using Corollary~\ref{global-sec-base-change-compat}, we get exact sequences 
$$0 \to \m_Q^{k + 1}\cdot L^\bullet(S_{k + 1}) \to L^\bullet(S_{k + 1}) \to L^\bullet(S_k) \to 0$$
on the level of global sections where the right hand map is compatible with $d$ and $\ell$. The limit 
$$L^\bullet := L^\bullet_{X_0/S_0} := \varprojlim L^\bullet(S_k)$$
is an $\A$-pdgla since its pieces are complete by \cite[09B8]{stacks} (note the index shift).

\begin{rem}
 In case $Q = 0$, the pdgla $L^\bullet$ is an actual dgla. Namely, we have $\A = k\llbracket 0\rrbracket = k$ and thus $\A_k = k$---all restrictions $\PV_{X_k/S_0} \to \PV_{X_0/S_0}$ are isomorphisms. Therefore, $d_k = d_0$ is a differential and $\ell(d_k) = 0$.
\end{rem}

\begin{prop}
 There is a natural equivalence $\mathbf{mc}: \mathrm{Def}_{L^\bullet} \Rightarrow TD_{X_0/S_0}$.
\end{prop}
\begin{proof}
 By \cite[09B8]{stacks} the canonical map 
 $$L^\bullet_{X_0/S_0} \otimes_\A \A_k \to L^\bullet(S_k)$$
 is an isomorphism, so after restriction along $S \to S_k$ and using Corollary~\ref{global-sec-base-change-compat}, we obtain an isomorphism 
 $$L^\bullet_{X_0/S_0} \otimes_\A \cO(S) \to L^\bullet(S) = \Gamma(X_0,\PV^{-1,\bullet}_{X_0/S}),$$
 which is functorial with respect to $S \to S_\ell$. Maurer--Cartan elements $\eta \in I_S \cdot L^1(S)$ induce differentials $d_\eta \in \mathrm{Diff}(\PV_{X_0/S})$ by Corollary~\ref{trivial-torsor}. If $\eta,\eta' \in I_S \cdot L^1(S)$ are gauge equivalent via $\theta \in I_S \cdot L^0(S)$, i.e., we have $d_\eta \circ \mathrm{exp}_\theta = \mathrm{exp}_\theta \circ d_{\eta'}$ on $L^\bullet(S)$, then $d_\eta,d_{\eta'}$ are gauge equivalent on $\PV_{X_0/S}$. Indeed, by Corollary~\ref{trivial-torsor} we can find a unique $\xi \in I_S\cdot L^1(S)$ such that 
 $$d_\xi = \mathrm{exp}_\theta \circ d_{\eta'} \circ \mathrm{exp}_\theta^{-1}$$
 on $\PV_{X_0/S}$. Restricting the equation to $L^\bullet(S) \subset \PV_{X_0/S}$, we find $\eta = \xi$ by Lemma~\ref{injective-G}. We get a well-defined transformation $\mathbf{mc}$ on the level of objects. It is injective because every automorphism of $\PV_{X_0/S}$ is a gauge transform, and it is surjective by Corollary~\ref{trivial-torsor} and the fact that $d_\eta$ is a differential if and only if $\eta$ is Maurer--Cartan. $\mathbf{mc}$ is a natural transformation, i.e., compatible with maps $S \to S'$, because we can always find a restriction $\PV_{X_0/S'} \to \PV_{X_0/S}$ which fits into a commutative diagram 
 \[
  \xymatrix{
   L^\bullet_{X_0/S_0} \otimes_\A \cO(S') \ar[r]^-\cong \ar[d] & \PV^{-1,\bullet}_{X_0/S'} \ar[d]\\
   L^\bullet_{X_0/S_0} \otimes_\A \cO(S) \ar[r]^-\cong & \PV^{-1,\bullet}_{X_0/S} \\
  }
 \]
 with the chosen horizontal maps. Finally, recall that on $TD_{X_0/S_0}$ the restriction is independent of the choice for $\PV_{X_0/S'} \to \PV_{X_0/S}$.
\end{proof}

Using Proposition~\ref{func-elem-prop}, we recover that $LD_{X_0/S_0}$ is a deformation functor, which was first proved in \cite{Kato1996}. We find that the tangent space $t_{LD}$ is isomorphic to $H^1(L^\bullet_0) \cong H^1(X_0, \Theta^1_{X_0/S_0})$, so in case $f_0: X_0 \to S_0$ is proper, we recover that $LD_{X_0/S_0}$ has a hull by Schlessingers' criterion \cite[Thm.~2.11]{Schlessinger1968}.

\section{Example: A Log Smooth Curve}\label{ex-sec}

In this section, we compute the pdgla $L^\bullet_{C_0/S_0}$ for a log smooth curve $f_0: C_0 \to S_0$; though not all intricacies of $L^\bullet_{X_0/S_0}$ show up here, it gives a first impression of how the pdgla looks like. As underlying space, we take 
$$\underline C_0 := \{XYZ = 0\} \subset \PP^2_k,$$
i.e., three pairwise intersecting lines.
It carries an obvious log structure of embedding type (see \cite[11.4]{Kato1996}) via the embedding $\underline C_0 \subset \PP^2_k$; however, it does not admit a global section (at least we cannot find it) which would yield a log smooth log morphism $f_0: C_0 \to S_0 := \Spec (\NN \to k)$. Instead, we turn $\underline C_0$ into a log smooth curve as follows: The standard open covering of $\PP^2_k$ induces an open covering of $\underline C_0$ with three opens 
$$V_x = \Spec k[y,z]/(yz), \quad V_y = \Spec k[z,x]/(zx), \quad V_z = \Spec k[x,y]/(xy)$$
which are glued via 
\begin{align}
 \cO(V_x|_{xy}) &= k[y]_y \xrightarrow{\cong} k[x]_x = \cO(V_y|_{xy}), \quad y \mapsto 1/x, \nonumber \\
 \cO(V_y|_{yz}) &= k[z]_z \xrightarrow{\cong} k[y]_y = \cO(V_z|_{yz}), \quad z \mapsto 1/y, \nonumber \\
 \cO(V_z|_{zx}) &= k[x]_x \xrightarrow{\cong} k[z]_z = \cO(V_x|_{zx}), \quad x\mapsto 1/z. \nonumber 
\end{align}
Each open $V_\alpha$ has a deformation $C_\alpha \to \bSS := \bAA^1_t$; more precisely, we have three log smooth families
\begin{align}
 C_x &:= \Spec (\NN^2 \to k[y,z]) \to \Spec (\NN \to k[t]) = \bSS, \quad t \mapsto yz, \nonumber \\
 C_y &:= \Spec (\NN^2 \to k[z,x]) \to \Spec (\NN \to k[t]) = \bSS, \quad t \mapsto zx, \nonumber \\
 C_z &:= \Spec (\NN^2 \to k[x,y]) \to \Spec (\NN \to k[t]) = \bSS, \quad t \mapsto xy. \nonumber 
\end{align}
We define $C_x|_{xy} := \{y \not= 0\} \subset C_x$ etc.---then we have isomorphisms
\begin{align}
 C_x|_{xy} &\xrightarrow{\cong} C_y|_{xy}, \quad y \mapsto 1/x, \quad z \mapsto zx^2, \nonumber \\
 C_y|_{yz} &\xrightarrow{\cong} C_z|_{yz}, \quad z \mapsto 1/y, \quad x \mapsto xy^2, \nonumber \\
 C_z|_{zx} &\xrightarrow{\cong} C_x|_{zx}, \quad x \mapsto 1/z, \quad y \mapsto yz^2 \nonumber 
\end{align}
of families of log schemes. These are not the standard isomorphisms which yield $\PP^2$; in fact, they do \emph{not} satisfy the cocycle condition on the triple intersection, and thus, they cannot be glued to a family. However, unlike the standard isomorphisms, they are compatible with the map to the base. Thus, after restriction to $\underline C_0$, they glue to a log smooth curve $f_0: C_0 \to S_0$ because on $\underline C_0$, the cocycle condition is empty---there is no triple intersection.

We compute the Gerstenhaber algebra $G^\bullet_{C_0/S_0}$ of polyvector fields; it is concentrated in degrees $0$ and $-1$. The sheaf of absolute log differentials $\Omega^1_{C_0}$ is on the open cover $\{V_\alpha\}$ given by 
$$\Omega^1_{C_0}|_{V_x} = \cO_{V_x} \left\langle \frac{dy}{y},\frac{dz}{z} \right\rangle , \quad
\Omega^1_{C_0}|_{V_y} = \cO_{V_y} \left\langle \frac{dz}{z},\frac{dx}{x} \right\rangle , \quad
\Omega^1_{C_0}|_{V_z} = \cO_{V_z} \left\langle \frac{dx}{x},\frac{dy}{y} \right\rangle
$$
where $\cO\langle a,b\rangle := \cO \cdot a \oplus \cO \cdot b$ is the direct sum on formal generators specified in the bracket; the transition maps are given by 
\begin{align}
 V_x|_{xy} &\to V_y|_{xy}: \qquad \frac{dy}{y} \mapsto -\frac{dx}{x}, \quad \frac{dz}{z} \mapsto 2\frac{dx}{x} + \frac{dz}{z} ,\nonumber \\
 V_x|_{zx} &\to V_z|_{zx}: \qquad \frac{dy}{y} \mapsto 2\frac{dx}{x} + \frac{dy}{y}, \quad \frac{dz}{z} \mapsto -\frac{dx}{x} ,\nonumber \\
 V_y|_{yz} &\to V_z|_{yz}: \qquad \frac{dx}{x} \mapsto \frac{dx}{x} + 2\frac{dy}{y}, \quad \frac{dz}{z} \mapsto - \frac{dy}{y} .\nonumber 
\end{align}
The image of $f_0^*\Omega^1_{S_0} \to \Omega^1_{C_0}$ is freely generated by the sum of the two generators on every $V_\alpha$, e.g.~$\frac{dy}{y} + \frac{dz}{z}$ on $V_x$; setting 
$$\omega|_{V_x} = \left[\frac{dy}{y}\right], \quad \omega|_{V_y} = \left[\frac{dz}{z}\right], \quad \omega|_{V_z} = \left[\frac{dx}{x}\right],$$
we get a global section $\omega \in \Gamma(C_0,\Omega^1_{C_0/S_0})$ of the quotient, which defines a trivialization $\cO_{C_0} \cong \Omega^1_{C_0/S_0}, \ 1 \mapsto \omega$, of the log canonical bundle. In particular, the relative tangent bundle $G^{-1}_{C_0/S_0} = \Theta^1_{C_0/S_0}$ is trivial as well; under the natural embedding $\Theta^1_{C_0/S_0} \subset \Theta^1_{\underline C_0}$ (which forgets the log part of the derivation and the fact that it is relative), the dual of $\omega$ is the derivation given by 
\begin{align}
 V_x: &\quad \partial_{yz} := y\partial_y - z\partial_z, \nonumber \\
 V_y: &\quad \partial_{zx} := z\partial_z - x\partial_x, \nonumber \\
 V_z: &\quad \partial_{xy} := x\partial_x - y\partial_y. \nonumber 
\end{align}
$\wedge$-multiplication with an element of $G^0_{C_0/S_0} = \cO_{C_0}$ is given by the $\cO_{C_0}$-module structure; since $G^{-2}_{C_0/S_0} = 0$, this determines the $\wedge$-multiplication completely. When $f,g \in \cO_{C_0}$ and $\theta \in \Theta^1_{C_0/S_0}$, then $[f,g] = 0$ and $[\theta,\theta] = 0$---note that there is two times $\theta$ in the bracket; in general, $[\theta,\xi] \not= 0$ for a second element $\xi \in \Theta^1_{C_0/S_0}$.  Finally, we have
\begin{align}
 &V_x: \qquad [\partial_{yz},y] = -y, \quad [\partial_{yz},z] = z, \nonumber \\
 &V_y: \qquad [\partial_{zx},z] = -z, \quad [\partial_{zx},x] = x, \nonumber \\
 &V_z: \qquad [\partial_{xy},x] = -x, \quad [\partial_{xy},y] = y \nonumber
\end{align}
since our Lie bracket is the negative of the Schouten--Nijenhuis bracket.

When $S \in \mathbf{Art}_\NN$, then the families $C_\alpha/\bSS$ induce log smooth deformations $V_\alpha/S$ as required at the beginning of Section~\ref{def-gerstenhaber-algebra-sec}. The corresponding Gerstenhaber algebras are 
\begin{align}
 G^\bullet_{V_x/S} &= \cO_{V_x/S} \cdot 1 \oplus \cO_{V_x/S} \cdot \partial_{yz}, \nonumber \\
 G^\bullet_{V_y/S} &= \cO_{V_y/S} \cdot 1 \oplus \cO_{V_y/S} \cdot \partial_{zx}, \nonumber \\
 G^\bullet_{V_z/S} &= \cO_{V_z/S} \cdot 1 \oplus \cO_{V_z/S} \cdot \partial_{xy} \nonumber
\end{align}
where $\cO_{V_x/S}$ is the structure sheaf of $C_x \times_\bSS S$---i.e., of the deformation $V_x/S$---and where $\partial_{yz}$ must be interpreted in the respective ring. Our isomorphisms $C_x|_{xy} \cong C_y|_{xy}$ etc.~above induce the required isomorphisms $\Phi_{\alpha\beta}: V_\beta|_{V_\alpha \cap V_\beta} \cong V_\alpha|_{V_\alpha \cap V_\beta}$. Because the cocycle condition is empty again, we can actually glue the pieces to a Gerstenhaber deformation $G^\bullet_{C_0/S} := \cO_{C_0/S} \cdot 1 \oplus \cO_{C_0/S} \cdot \partial$ in the sense of Definition~\ref{gerstenhaber-defo-defn}.

Next, we perform the Thom--Whitney resolution of $G^\bullet_{C_0/S}$ with respect to the cover $\U := \{V_\alpha\}$; this yields a TWG deformation together with a differential on it, not only a predifferential. In particular, this computes $\PV_{C_0/S} \cong \TW^{\bullet,\bullet}(G^\bullet_{C_0/S}(\U))$. Since the Gerstenhaber deformations $G^\bullet_{C_0/S}$ are compatible for all $S \in \mathbf{Art}_\NN$, we also get restriction maps between the different $\PV_{C_0/S}$ and compatible differentials. Thus, in order to obtain 
$$L^\bullet_{C_0/S_0} = \varprojlim \Gamma(C_0,\PV^{-1,\bullet}_{C_0/S_k}),$$
we only need to compute the global sections of the Thom--Whitney resolution.

After ordering the index set as $x < y < z$, the global sections of the semicosimplicial sheaf of Gerstenhaber algebras $G^\bullet_{C_0/S}(\U)$ are
\[
 \xymatrix{
  \cO(V_x) \times \cO(V_y) \times \cO(V_z) \ar@<-.27em>[r]_-{\delta_1} \ar@<+.27em>[r]^-{\delta_0} & \cO(V_y \cap V_z) \times \cO(V_z \cap V_x) \times \cO(V_x \cap V_y) \\
  \Theta^1(V_x) \times \Theta^1(V_y) \times \Theta^1(V_z) \ar@<-.27em>[r]_-{\delta_1} \ar@<+.27em>[r]^-{\delta_0} & \Theta^1(V_y \cap V_z) \times \Theta^1(V_z \cap V_x) \times \Theta^1(V_x \cap V_y) \\
 }
\]
with $\delta_0(a,b,c) = (c,c,b)$ and $\delta_1(a,b,c) = (b,a,a)$ (in both rows); all other pieces $G^i_n := G^i_{C_0/S}(\U)_n$ are zero. Concretely, for $S_k = \Spec (\NN \to k[t]/(t^{k + 1}))$, we get 
\begin{align}
 k[y,z,t]/(yz - t, t^{k + 1}) \times k[z,x,t]/(zx - t, t^{k + 1}) \times k[x,y,t]/(xy - t, t^{k + 1}) \nonumber \\
 \substack{\longrightarrow\\[-0.9em] \longrightarrow} \
 k[z,t]_z/(t^{k + 1}) \times k[x,t]_x/(t^{k + 1}) \times k[y,t]_y/(t^{k + 1}) \nonumber \\
 \nonumber \\
 k[y,z,t]/(yz - t, t^{k + 1}) \cdot \partial \times k[z,x,t]/(zx - t, t^{k + 1}) \cdot \partial \times k[x,y,t]/(xy - t, t^{k + 1}) \cdot \partial \nonumber \\
 \substack{\longrightarrow\\[-0.9em] \longrightarrow} \
 k[z,t]_z/(t^{k + 1}) \cdot \partial \times k[x,t]_x/(t^{k + 1}) \cdot \partial \times k[y,t]_y/(t^{k + 1}) \cdot \partial \nonumber
\end{align}
with 
$$\delta_0(y,z,x) = (tz,x,ty), \qquad \delta_1(y,z,x) = (z,tx,y),$$
$$\delta_0(z,x,y) = (1/z,t/x,1/y), \qquad \delta_1(z,x,y) = (t/z,1/x,t/y)$$
and the same for the $\Theta^1$-row. The relevant pieces of the semisimplicial dgca $A_{PL}$ are 
$$(A_{PL})_0^0 = k, \ (A_{PL})_1^0 \cong k[t_0], \ (A_{PL})_0^1 = 0, \ (A_{PL})_1^1 \cong k[t_0]dt_0;$$
the two non-trivial relevant face maps $(A_{PL})_1^0 \to (A_{PL})_0^0$ are $t_0 \mapsto 0$ and $t_0 \mapsto 1$. We find 
\begin{align}
 \TW^{0,0}_k &:= \TW^{0,0}(G^\bullet_{C_0/S_k}(\U)) \nonumber \\
 &= \{((f,g,h),\sum t_0^i(p_i,q_i,r_i)) \in k \otimes G^0_0 \times k[t_0] \otimes G_1^0 \ | \ \nonumber \\
 &\quad \delta_0(f,g,h) = (p_0,q_0,r_0), \ \delta_1(f,g,h) = \sum_i (p_i,q_i,r_i)\} \nonumber \\
 &\cong G_0^0 \times t_0^2(k[t_0] \otimes G_1^0); \nonumber
\end{align}
the last isomorphism is due to the fact that we can choose the $G^0_0$-part and the $t_0^i$-parts for $i \geq 2$ arbitrarily whereafter the remaining two parts are uniquely determined by the equations. To write down the limit $\TW^{0,0}_\infty := \varprojlim \TW^{0,0}_k$, let us introduce notation
$$k\llangle x,y\rrangle := \left\{\sum_{i,j = 0}^\infty \alpha_{ij}x^iy^j \ | \ \forall k : \#\{\alpha_{ij} | \alpha_{ij} \not= 0 \wedge (i < k \vee j < k)\} < \infty \right\} \subset k\llbracket x,y\rrbracket;$$
then we have 
$$\TW^{0,0}_\infty \cong k\llangle y,z \rrangle \times k\llangle z,x\rrangle \times k\llangle x,y\rrangle \times t_0^2k[t_0,z]_z\llbracket t\rrbracket \times t_0^2k[t_0,x]_x\llbracket t\rrbracket \times t_0^2k[t_0,y]_y\llbracket t\rrbracket .$$
Since $(A_{PL})_0^1 = 0$, the next piece is 
$$\TW^{0,1}_\infty \cong k[t_0,z]_z\llbracket t\rrbracket dt_0 \times k[t_0,x]_x\llbracket t\rrbracket dt_0 \times k[t_0,y]_y\llbracket t\rrbracket dt_0;$$
finally,
$$\TW^{-1,0}_\infty = \TW^{0,0}_\infty \cdot \partial, \quad \TW^{-1,1}_\infty = \TW^{0,1}_\infty \cdot \partial . $$
By construction, ($\TW^{0,0}_\infty, \wedge)$ is canonically a subring of 
$$k\llangle y,z \rrangle \times k\llangle z,x\rrangle \times k\llangle x,y\rrangle \times k[t_0,z]_z\llbracket t\rrbracket \times k[t_0,x]_x\llbracket t\rrbracket \times k[t_0,y]_y\llbracket t\rrbracket ;$$
if $(f,g,h,p,q,r) \in \TW^{0,0}_\infty$ with $p = \sum p_i(t,z)t_0^i$ etc., then in the image $(p_0,q_0,r_0) = \delta_0(f,g,h)$ and $(p_1,q_1,r_1) = \delta_1(f,g,h) - \delta_0(f,g,h) - \sum_{i \geq 2}(p_i,q_i,r_i)$. In particular, the unit of the ring is $(1,1,1,0,0,0)$; now, $\TW^{-1,0}_\infty$ is a free $\TW^{0,0}_\infty$-module of rank $1$ with generator $\partial := 1 \cdot \partial$. The operator $[\partial,-]$ acts as a derivation on $\TW^{0,0}_\infty$, and it acts on $\TW^{0,1}_\infty$. Via this action, we describe the Lie bracket on the graded Lie algebra $L^\bullet_{C_0/S_0} = \TW^{-1,\bullet}_\infty$; namely, for $a,b \in \TW^{0,0}_\infty$, we have 
$$[a\partial,b\partial] = (a[\partial,b] - b[\partial,a])\partial$$
for the Lie bracket $[\TW^{-1,0}_\infty,\TW^{-1,0}_\infty]$,
and for $a \in \TW^{0,0}_\infty,m \in \TW^{0,1}_\infty$, we have
$$[a\partial,m\partial] = (a[\partial,m] - [\partial,a]m)\partial$$
for the Lie bracket $[\TW^{-1,0}_\infty,\TW^{-1,1}_\infty]$; moreover, $[\TW^{-1,1}_\infty,\TW^{-1,1}_\infty] = 0$ since $\TW^{-1,2}_\infty = 0$. The differential $d: \TW^{-1,0}_\infty \to \TW^{-1,1}_\infty$ is given by formally deriving $t_0$, i.e., by mapping $ft_0 \mapsto fdt_0$ etc.---after embedding $\TW^{-1,0}_\infty = \TW^{0,0}_\infty \cdot \partial$ in 
$$\left(k\llangle y,z \rrangle \times k\llangle z,x\rrangle \times k\llangle x,y\rrangle \times k[t_0,z]_z\llbracket t\rrbracket \times k[t_0,x]_x\llbracket t\rrbracket \times k[t_0,y]_y\llbracket t\rrbracket\right) \cdot \partial ,$$
only the last three factors matter. On e.g.~the factor $k[t_0,z]_z\llbracket t\rrbracket$, the map is given by $\sum p_i(t,z)t_0^i \mapsto \sum ip_i(t,z)t_0^{i - 1}$. Since $d$ is a differential, we have $\ell = 0$; the $k\llbracket t\rrbracket$-linear pdgla $L^\bullet_{C_0/S_0}$ is actually a dgla in this example.

Since $L^2_{C_0/S_0} = 0$, every $\eta \in L^1_{C_0/S_0} \otimes_{\llbracket t \rrbracket} A$ solves the Maurer--Cartan equation; we leave it to the interested reader to compute the gauge equivalence classes and the deformation functor. Conjecturally, it is represented by a $k\llbracket t\rrbracket$-algebra which is isomorphic to $k\llbracket t,s\rrbracket$. Namely, $f_0: C_0 \to S_0$ is log smooth and log Calabi--Yau, so we expect the deformation functor to be represented by a smooth $k\llbracket t\rrbracket$-algebra; its fiber over $k$ should have as many variables as the tangent space $H^1(L^\bullet_{C_0/S_0} \otimes_{k\llbracket t\rrbracket} k) \cong H^1(C_0,\Theta^1_{C_0/S_0})$ has dimensions.

\section{Appendix}

\subsection{The Strict Locus of Log Smooth Morphisms}

We include the result below for which we have, for the general case, no published reference. For an $s$-injective morphism $f: X \to S$ of fs log schemes---$s$-injective means that the induced map $\overline\M_{S,f(\bar x)} \to \overline\M_{X,\bar x}$ on ghost stalks is injective---consider the \emph{strict locus} $X^{str} \subset X$, i.e., the union of all opens $U \subset X$ such that $f|_U: U \to S$ is strict. It is the maximal open subset with that property. Since $X,S$ are fine, $f$ is strict on $U$ if and only if $\phi: f^{-1}(\overline\M_S) \to \overline\M_X$ is an isomorphism on $U$. For a geometric point $\bar x \in X$ where $\phi_{\bar x}$ is an isomorphism on stalks, $\phi$ is surjective in an \'etale neighbourhood of $\bar x$ due to coherence (since there every stalk is a quotient of $\overline\M_{X,\bar x}$). $s$-injectivity shows it is an isomorphism. Thus $\bar x \in X^{str}$, and the converse is clear. We see that the formation of $X^{str}$ commutes with base change along strict morphisms $T \to S$.

\begin{prop}\label{strict-dense}
 Let $f: X \to S$ be a log smooth and saturated morphism of fs log schemes. Then $X^{str} \subset X$ is scheme-theoretically dense. In particular, the forgetful map $\Theta^1_{X/S} \to \Theta^1_{\underline X/\underline S}$ to the derivations of underlying schemes is injective.
\end{prop}
\begin{proof}
 First note that $f: X \to S$ is exact and hence $s$-injective. Using \cite[III.~Thm.~3.3.3]{LoAG2018}, we see that it is sufficient to prove the density statement for the morphism 
 $$\Spec (P \to \ZZ[P]) \to \Spec (Q \to \ZZ[Q])$$ induced by a saturated injection $Q \subset P$ of sharp toric monoids. Now we can take the set $U_2$ of \cite[Cor.~3.11]{FFR2019}. 
\end{proof}
\begin{rem}\label{sat-dense-rem}
 The density statement fails if we assume $f: X \to S$ only integral, but not saturated. Consider e.g.~$\Spec(\NN \to \CC[t]/(t^2)) \to \Spec(\NN \to \CC)$ induced by $\NN \to \NN, 1 \mapsto 2$, which is log smooth and integral, but nowhere strict. Moreover, we do not know if the injectivity statement remains true. This is the reason why we restrict the whole paper to \emph{saturated} log smooth morphisms.
\end{rem}
\begin{rem}
 In the case $S = \Spec (0 \to k)$, Proposition~\ref{strict-dense} is (in a weak sense) equivalent to \cite[Prop.~2.6]{Niziol2006}. Namely, if $f: X \to S$ is log smooth and saturated, then it is log-regular, and the \emph{log-trivial} locus $X^{tr}$ equals the strict locus $X^{str}$. The former is dense by \cite[Prop.~2.6]{Niziol2006}, so the latter is dense; since $X$ is reduced, $X^{str}$ is scheme-theoretically dense as well. Conversely, if $X$ is a log-regular fs log scheme of finite type over $S$, then the (structure) morphism $f: X \to S$ is saturated and log smooth by \cite[IV, Thm.~3.5.1]{LoAG2018}. Proposition~\ref{strict-dense} implies that $X^{str} \subset X$ is scheme-theoretically dense, so $X^{tr} = X^{str}$ is dense.
\end{rem}

\subsection{$\Lambda$-Linear Predifferential Graded Lie Algebras}\label{pdgla}

Fix a complete local Noetherian ring $\Lambda$ with maximal ideal $\m \subset \Lambda$. We briefly introduce the type of Lie algebra which controls log smooth deformations. This type of Lie algebra is called almost dgla in \cite[Thm.~1.1]{ChanLeungMa2019} where the existence of such a dgla is proven under some abstract conditions.

\begin{defn} A \emph{$\Lambda$-linear predifferential graded Lie algebra} ($\Lambda$-pdgla) is a quadruple $(L^\bullet,[-,-],d,\ell)$ where $(L^\bullet,[-,-])$ is a graded $\Lambda$-linear Lie algebra such that every $L^i$ is complete, $d: L^\bullet \to L^{\bullet + 1}$ is a $\Lambda$-linear derivation of degree $1$ and $\ell \in \m\cdot L^2$ is such that $d^2 = [\ell,-]$. In particular, $d$ is \emph{not} a differential in general. 
\end{defn}

For a local Artin $\Lambda$-algebra $A$ (with residue field $k := \Lambda/\m\Lambda$), the tensor product $L^\bullet \otimes_\Lambda A$ is an $A$-pdgla since $L^i \otimes_\Lambda A$ is complete. We consider them as some sort of infinitesimal deformation of the dgla 
$$L^\bullet_0 := L^\bullet \otimes_\Lambda k,$$
which we call the central fiber. In fact, since $\ell \in \m\cdot L^2$, it is a dgla ($d^2 = 0$).
We say $L^\bullet$ is \emph{faithful} if for all $A \in \mathbf{Art}_\Lambda$ and all $\xi \in L^\bullet \otimes_\Lambda A$, we have that $[\xi,-] = 0$ implies $\xi = 0$. In a faithful $L^\bullet$, the equality $[\xi,-] = 0$ implies $\xi = 0$ also for $\xi \in L^\bullet$.

Elements $\theta \in \m\cdot L^0$ give rise to \emph{gauge transforms}
$$\mathrm{exp}_\theta: L^\bullet \to L^\bullet, \quad \xi \mapsto \sum_{i = 0}^\infty\frac{([\theta,-])^i(\xi)}{i!} = \xi + [\theta,\xi] + \frac{1}{2}[\theta,[\theta,\xi]] + ...,$$
which are well-defined since $L^\bullet$ is complete. We find the identities 
$$\mathrm{exp}_\theta(\xi + \chi) = \mathrm{exp}_\theta(\xi) + \mathrm{exp}_\theta(\chi), \quad 
\mathrm{exp}_\theta([\xi,\chi]) = [\mathrm{exp}_\theta(\xi),\mathrm{exp}_\theta(\chi)]$$
but $\mathrm{exp}_\theta$ is not compatible with $d$, i.e., in general $\mathrm{exp}_\theta \circ d \not= d \circ \mathrm{exp}_\theta$. If $L^\bullet$ is faithful, then $\mathrm{exp}_\theta = \mathrm{exp}_{\theta'}$ implies $\theta = \theta'$. We consider gauge transforms as some sort of infinitesimal automorphism. In fact, $\mathrm{exp}_\theta$ induces the identity on $L^\bullet_0$.

Given an element $\eta \in \m\cdot L^1$, we define a derivation $d_\eta := d + [\eta,-]$ and consider it as a deformation of the differential on $L_0^\bullet$. We have $d_\eta^2 = 0$ if and only if $\eta$ satisfies the \emph{Maurer--Cartan equation}
$$d\eta + \frac{1}{2}[\eta,\eta] + \ell = 0.$$
In this case, we say $\eta$ is \emph{Maurer--Cartan}. If $L^\bullet$ is faithful, then $d_\eta = d_{\eta'}$ implies $\eta = \eta'$. We say two elements $\eta,\eta'$ are \emph{gauge equivalent} if there is $\theta \in \m\cdot L^0$ with $d_\eta \circ \mathrm{exp}_\theta = \mathrm{exp}_\theta \circ d_{\eta'}$. In this case, $\eta$ is Maurer--Cartan if and only if $\eta'$ is. We consider gauge equivalence classes of Maurer--Cartan elements as deformations of $L^\bullet_0$. We define the
\emph{Maurer--Cartan functor} 
$$\mathrm{MC}_{L^\bullet}: \mathbf{Art}_\Lambda \to \mathbf{Set}$$
by taking $A \in \mathbf{Art}_\Lambda$ to the set of Maurer--Cartan elements in $L^\bullet \otimes_\Lambda A$
and the \emph{deformation functor}
$$\mathrm{Def}_{L^\bullet}: \mathbf{Art}_\Lambda \to \mathbf{Set}$$
by taking gauge equivalence classes thereof. In case $\Lambda = k$ (in which $L^\bullet$ is an actual dgla) these functors reduce to the classical Maurer--Cartan and deformation functor. In general, they share a number of properties with them:
\begin{prop}\label{func-elem-prop}
 We have:
 \begin{itemize}
  \item The Maurer--Cartan functor $\mathrm{MC}_{L^\bullet}$ is homogeneous.
  \item The transformation $\mathrm{MC}_{L^\bullet} \Rightarrow \mathrm{Def}_{L^\bullet}$ is smooth and surjective.
  \item The functor $\mathrm{Def}_{L^\bullet}$ is a deformation functor.
  \item The tangent space of $\mathrm{MC}_{L^\bullet}$ is isomorphic to $Z^1(L^\bullet_0)$.
  \item The tangent space of $\mathrm{Def}_{L^\bullet}$ is isomorphic to $H^1(L^\bullet_0)$.
 \end{itemize}
\end{prop}
\begin{proof}
 The proof does not employ new ideas beyond the classical situation (see e.g.~\cite{Manetti1999}), so we give only a brief indication. Let $A' \to A$ and $A'' \to A$ be surjections in $\mathbf{Art}_\Lambda$. We consider the canonical map 
 $$\sigma: F(A'' \times_A A') \to F(A'') \times_{F(A)} F(A')$$
 for any functor $F$ of Artin rings.
 Recall that homogeneous means $\sigma$ is an isomorphism. It suffices to prove this for a small extension $A'' \to A$. In this case, $B := A'' \times_A A' \to A'$ is a small extension as well with the same kernel $I$, so we obtain a diagram 
 \[
  \xymatrix{
   0 \ar[r] & L_0^\bullet \otimes_k I \ar[r] \ar@{=}[d] &  L^\bullet \otimes_\Lambda B \ar[d] \ar[r] & L^\bullet \otimes_\Lambda A' \ar[r] \ar[d] & 0 \\
   0 \ar[r] & L_0^\bullet \otimes_k I \ar[r] &  L^\bullet \otimes_\Lambda A'' \ar[r] & L^\bullet \otimes_\Lambda A \ar[r] & 0 \\
  }
 \]
 with exact rows. When we use the fact that the two kernels are equal, a straightforward diagram chase yields the result. To prove $\mathrm{MC}_{L^\bullet} \Rightarrow \mathrm{Def}_{L^\bullet}$ smooth, we use that every gauge transform can be lifted to some gauge transform because $L^0 \otimes A' \to L^0 \otimes A$ is surjective. To prove that $\mathrm{Def}_{L^\bullet}$ is a deformation functor, we use the map $\mathrm{MC}_{L^\bullet} \Rightarrow \mathrm{Def}_{L^\bullet}$ and that, when $A = k$, the restriction of a gauge transform on $L^\bullet \otimes A'$ to $A$ is trivial.
\end{proof}
\begin{rem}
 The functor $\mathrm{MC}_{L^\bullet}$ is a deformation functor (like every homogeneous functor), but $\mathrm{Def}_{L^\bullet}$ is \emph{not} homogeneous. Namely, a deformation functor is prorepresentable if and only if it is homogeneous and has finite-dimensional tangent space, but e.g.~the flat deformation functor of some surfaces is not prorepresentable (which is certainly an example of our theory).
\end{rem}

If we replace $d$ by $d_\eta$ and $\ell$ by the corresponding $\ell_\eta$, then the deformation functor remains unchanged. Taking this into account, we propose the following notion of homomorphism for future study of pdglas.

\begin{defn}
 A \emph{homomorphism} $\psi: L^\bullet \to M^\bullet$ of $\Lambda$-pdglas is a $\Lambda$-linear homomorphism of graded Lie algebras together with $\kappa_\psi \in \m\cdot M^1$ such that $d_M \circ \psi - \psi \circ d_L = [\kappa_\psi,\psi(-)]$ and $\psi(\ell_L) = \ell_M - d_M\kappa_\psi + \frac{1}{2}[\kappa_\psi,\kappa_\psi]$.
\end{defn}

A homomorphism induces a map $\mathrm{Def}(\psi)$ sending 
$\eta \mapsto \psi(\eta) - \kappa_\psi$.

\bibliography{Gerstenhaber-arXiv-II.bib}
\bibliographystyle{plain}

\end{document}